\documentclass[11pt]{amsart}

\usepackage{amsmath}
\usepackage{amsfonts,amssymb,enumerate,color,bm,hhline,makecell,array}
\usepackage{array}
\usepackage{mathrsfs}
\usepackage{comment}
\usepackage[all,ps,cmtip]{xy} 
\usepackage[normalem]{ulem}
\usepackage[colorlinks=true,citecolor=blue, urlcolor=blue, linkcolor=blue, pagebackref]{hyperref}
\usepackage{tikz-cd}
\usepackage{amscd}
\usepackage[shortlabels]{enumitem}
\usepackage{tensor}
\usepackage{tikz}
\usetikzlibrary{matrix,arrows}
\usepackage{extarrows}
\usepackage{stmaryrd} 
\usepackage[new]{old-arrows} 

\usepackage{graphicx}

\usepackage{mathtools}

\usepackage{graphicx}

\newcommand{\CC}{{\mathbb C}}

\newcommand{\NN}{{\mathbb N}}
\newcommand{\PP}{{\mathbb P}}
\newcommand{\QQ}{{\mathbb Q}}
\newcommand{\RR}{{\mathbb R}}
\newcommand{\ZZ}{{\mathbb Z}}

\let\mathcal\mathscr

\def\cA{\mathcal{A}}
\def\cB{\mathcal{B}}

\def\cD{\mathcal{D}}
\def\cE{\mathcal{E}}
\def\cF{\mathcal{F}}
\def\cG{\mathcal{G}}
\def\cH{\mathcal{H}}
\def\cI{\mathcal{I}}

\def\cK{\mathcal{K}}
\def\cL{\mathcal{L}}
\def\cM{\mathcal{M}}

\def\cO{\mathcal{O}}
\def\cP{\mathcal{P}}
\def\cQ{\mathcal{Q}}

\def\cS{\mathcal{S}}

\def\cU{\mathcal{U}}

\def\cX{\mathcal{X}}

\def\bD{\ensuremath{\mathbf D}}

\def\gotg{\mathfrak g}

\def\goth{\mathfrak h}
\def\gotl{\mathfrak l}

\def\goto{\mathfrak o}

\def\gots{\mathfrak s}

\def\sF{{\mathsf F}}
\def\sG{{\mathsf G}}
\def\sP{{\mathsf P}}

\def\phi{{\varphi}}

\DeclareMathOperator{\Amp}{Amp}
\DeclareMathOperator{\an}{an}
\DeclareMathOperator{\Ann}{Ann}
\DeclareMathOperator{\Aut}{Aut}

\DeclareMathOperator{\BKR}{BKR}

\DeclareMathOperator{\ch}{ch}
\DeclareMathOperator{\CH}{CH}
\DeclareMathOperator{\cl}{cl}

\DeclareMathOperator{\Coh}{Coh}

\DeclareMathOperator{\Def}{Def}

\DeclareMathOperator{\divisore}{div}

\DeclareMathOperator{\ext}{ext}
\DeclareMathOperator{\Ext}{Ext}

\DeclareMathOperator{\Gr}{Gr}
\newcommand{\GR}{\mathbb{G}\mathrm{r}}

\DeclareMathOperator{\Hom}{Hom}
\DeclareMathOperator{\Id}{Id}

\def\Im{\mathop{\rm Im}\nolimits}

\DeclareMathOperator{\ks}{ks}
\DeclareMathOperator{\Kum}{Kum}
\DeclareMathOperator{\KKK}{K3}

\DeclareMathOperator{\Nef}{Nef}
\DeclareMathOperator{\NS}{NS}

\DeclareMathOperator{\obs}{obs}

\DeclareMathOperator{\Pic}{Pic}

\DeclareMathOperator{\Rep}{Rep}

\DeclareMathOperator{\sss}{ss}
\DeclareMathOperator{\st}{st}

\DeclareMathOperator{\Spec}{Spec}

\DeclareMathOperator{\Sym}{Sym}

\DeclareMathOperator{\td}{td}
\DeclareMathOperator{\Tr}{Tr}

\newcommand{\es}{{\emptyset}}

\newcommand{\la}{\langle}

\newcommand{\ov}{\overline}
\newcommand{\ra}{\rangle}

\newcommand{\wt}{\widetilde}

\def\lra{\longrightarrow}
\def\llra{\hbox to 10mm{\rightarrowfill}}
\def\lllra{\hbox to 15mm{\rightarrowfill}}

\def\llla{\hbox to 10mm{\leftarrowfill}}
\def\lllla{\hbox to 15mm{\leftarrowfill}}

\def\hra{\hookrightarrow}

\newtheorem{lmm}{Lemma}[subsection]
\newtheorem{thm}[lmm]{Theorem}
\newtheorem*{MainThm*}{Main Theorem}
\newtheorem{crl}[lmm]{Corollary}
\newtheorem{prp}[lmm]{Proposition}
\newtheorem*{clm}{Claim}

\theoremstyle{definition}
\newtheorem{dfn}[lmm]{Definition}
\newtheorem{rmk}[lmm]{Remark}
\newtheorem{cnj}[lmm]{Conjecture}

\newtheorem{expl}[lmm]{Example}

\newtheorem{hyp}[lmm]{Hypotheses}

\theoremstyle{remark}
\newtheorem{cmm}[lmm]{Comment}
\newtheorem*{cmm*}{Comment}
\newtheorem*{rmk*}{Remark}
\newtheorem*{lmmred*}{Lemma}

\makeatletter
\@namedef{subjclassname@2020}{%
  \textup{2020} Mathematics Subject Classification}
\makeatother

\makeatletter
\def\@tocline#1#2#3#4#5#6#7{\relax
  \ifnum #1>\c@tocdepth 
  \else
    \par \addpenalty\@secpenalty\addvspace{#2}%
    \begingroup \hyphenpenalty\@M
    \@ifempty{#4}{%
      \@tempdima\csname r@tocindent\number#1\endcsname\relax
    }{%
      \@tempdima#4\relax
    }%
    \parindent\z@ \leftskip#3\relax \advance\leftskip\@tempdima\relax
    \rightskip\@pnumwidth plus4em \parfillskip-\@pnumwidth
    #5\leavevmode\hskip-\@tempdima
      \ifcase #1
       \or\or \hskip 1em \or \hskip 2em \else \hskip 3em \fi%
      #6\nobreak\relax
    \dotfill\hbox to\@pnumwidth{\@tocpagenum{#7}}\par
    \nobreak
    \endgroup
  \fi}
\makeatother

\allowdisplaybreaks

\usepackage{enumitem}
\setlist[itemize]{noitemsep,nolistsep}
\setlist[enumerate]{noitemsep,nolistsep}
\usepackage{colortbl}
\usepackage{enumerate}

\title{Moduli of sheaves on hyperk\"ahler manifolds}

\begin{document} 

\author{Kieran G.~O'Grady}
\address{\parbox{0.9\textwidth}{Sapienza Universit\`a di Roma, Dipartimento di Matematica Guido Castelnuovo\\[1pt]
P.le A. Moro 5, 00185 Roma, Italia
\vspace{1mm}}}
\email{{ogrady@mat.uniroma1.it}}

\thanks{}

\begin{abstract}
We survey recent advances in the theory of moduli spaces of stable sheaves on hyperk\"ahler manifolds of dimension greater than $2$. We start by recalling the well-known theory in dimension $2$, i.e.~for $K3$ surfaces, emphasizing the techniques which can be extended to higher dimensions. 
\end{abstract}

\maketitle

\tableofcontents
\setcounter{tocdepth}{2}


\section{Introduction}\label{sec:intro}
\subsection{Background and goal of the paper}
\setcounter{equation}{0}
Algebraic vector bundles  appeared quite lately in the history of Algebraic Geometry, see~\cite{weil:fibre-spaces}. In retrospect algebraic line bundles had been present since Abel's proof~\cite{abel:hyperellint} of one half of the Abel-Jacobi Theorem, which  may be interpreted as a forerunner of the theory of moduli spaces of semistable sheaves on a polarized projective variety $(X,h)$ developed by Mumford, Narasimhan, Seshadri, Gieseker, Maruyama, Simpson (and others). The latter theory associates to  $(X,h)$ a plethora of moduli spaces of semistable sheaves on $(X,h)$, i.e.~projective schemes parametrizing isomorphism classes of stable sheaves on $X$  (and also  $S$-equivalence classes of semistable sheaves) with fixed numerical invariants. 
Associating to $(X,h)$ these moduli spaces  is  analogous to associating to a group $G$ the category $\Rep(G)$ of its  representations.
There are  (at least) two reasons to be interested in these moduli spaces. The first is that they give precious geometric information regarding $X$. 
The second is that  moduli spaces  often are very interesting  varieties  in their own right. 

What is the status of our knowledge of these moduli spaces? Suppose for simplicity that $X$ is smooth. If $X$ is curve then the moduli spaces are as nice as possible, meaning that they are non empty, irreducible of the expected dimension if the genus $g$ of $X$ is at least $2$ (if $g\le 1$ they are either empty or have very explicit descriptions), and many intriguing results are known about them.
If $X$ is a surface the moduli spaces may present various pathological phenomena, but
 general results are available. These state that if the expected dimension of the moduli space is large enough (equivalently: the sheaves in question are very twisted) then it is non-empty, irreducible, of the expected dimension, and generically smooth. For surfaces of non maximal Kodaira dimension more detailed information is available, notably for $K3$ or abelian surfaces.  There are no general results valid for $X$ of dimension greater than $2$. 

A Hyperk\"ahler (HK) manifold is a  simply connnected compact K\"ahler manifold $X$ carrying a holomorphic symplectic form whose class spans  
$H^0(X,\Omega^2_X)$. \emph{Two dimensional HK manifold} is sinonymous of \emph{$K3$ surface}. Higher dimensional HK manifolds behave very much like surfaces, and in some respects are \lq\lq surface-like\rq\rq.

In the last few years there have been several developments in the theory of semistable sheaves on higher dimensional HK manifolds, more precisely of \lq\lq modular sheaves\rq\rq\, to be soon introduced. The emerging theory    promises to be as intriguing and beautiful as that of sheaves on $K3$ surfaces. The goal of the present paper is to give an introduction to some of the methods and results of this theory.
\subsection{Outline of the paper}
\setcounter{equation}{0}
In Section~\ref{sec:kappatre} we recap some of the main results of the theory of moduli of sheaves on $K3$ surfaces. 
The results are the model for what one would like to achieve in higher dimensions. In addition, some of the methods that are employed in dimension $2$ can be extended to higher dimensions.
A third reason for going through the theory for $K3$ surfaces is that interesting stable vector bundles  on projective HK manifolds of Type $K3^{[n]}$ (i.e.~deformations of the Hilbert scheme $S^{[n]}$ parametrizing lenght-$n$ subschemes of a $K3$ surface $S$) are produced from vector bundles on $K3$ surfaces  via the Bridgeland-King-Reid (BKR) equivalence. 

In Section~\ref{sec:alfinlaciccia} we introduce and discuss the three main classes of sheaves on HK manifolds for which significant results have been proved, namely modular torsion-free sheaves (sheaves for which a certain codimension-$4$ characteristic class satisfies a purely topological condition),  projectively hyperholomorphic vector bundles, and atomic sheaves. Roughly speaking these classes are strictly decreasing, i.e.~an atomic stable vector bundle is projectively hyperholomorphic, and the latter is modular. 
Variation of slope (semi)stability
of modular sheaves   behaves as if they were sheaves on surfaces, and moreover stable modular sheaves on HK manifolds equipped with a Lagrangian fibration behave very much like stable sheaves on elliptic $K3$ surfaces. This motivates the belief that a theory resembling that valid for $K3$ surfaces can be developed in higher dimensions.
In Section~\ref{sec:alfinlaciccia} we give \lq\lq synthetic\rq\rq\ examples of modular sheaves on $S^{[n]}$, where
$S$ is a  $K3$ surface, via the BKR equivalence. 

In Section~\ref{sec:alfinteoremi} we state the main results that have been proved about moduli of stable projectively hyperholomorphic vector bundles on HK manifolds. Here we should  mention that if a slope-stable vector bundle $\cF$ on a HK manifold $X$ is projectively hyperholomorphic, then it extends to every deformation of $X$ keeping $c_1(\cF)$ of type $(1,1)$, and the projectivization $\PP(\cF)$ extends to all deformations of $X$. Every  slope-stable vector bundle on a $K3$ surface is projectively hyperholomorphic, while in higher dimensions 
a randomly chosen slope-stable vector bundle is not projectively hyperholomorphic. The  results that we present are all about   projectively hyperholomorphic vector bundles. 
The examples of Section~\ref{sec:alfinlaciccia} are chosen carefully, so that when they are slope stable they are projectively hyperholomorphic.
We present in some detail our results  on existence and uniqueness of slope-stable stable rigid vector bundles on general suitably polarized 
HK manifolds of Type $K3^{[n]}$. We quickly go over Bottini's work which produces HK manifolds of Type OG10 (one of the two \lq\lq sporadic\rq\rq\ deformation classes) as (an irreducible component of) a moduli space of stable (twisted) vector bundles on HK manifolds of Type $K3^{[2]}$. We  mention  Markman's work on (twisted) hyperholomorphic vector bundles on couples of HK manifolds  of Type $K3^{[n]}$: it has produced highly non-trivial results, e.g.~the proof (by Markman) of the Shafarevich conjecture for couples 
 of  HK manifolds (both of) of Type $K3^{[n]}$, and the proof by Maulik, Shen, Yin and Zhang of the $D$-equivalence conjecture for HK manifolds  of Type $K3^{[n]}$. Lastly we go over work in progress of ours whose goal is to show that one gets an arbitrary HK manifold  of Type $K3^{[a^2+1]}$ as the normalization of a (component of) of a suitable moduli space of (twisted) hyperholomorphic vector bundles on a HK manifold  of Type $K3^{[2]}$ (and also  
 to give the analogue of Markman's result above for couples 
$(X,Y)$ of  HK manifolds  of Type $K3^{[a^2+1]}$ and $K3^{[b^2+1]}$).
\subsection{Conventions and notation}
\setcounter{equation}{0}
We denote by $V_n,W_n,\ldots$ complex vector spaces of dimension $n$.

We work over the field of complex numbers.
Hence by scheme  we mean a complex scheme, etc. If $T$ is scheme $t\in T$  is a \emph{closed} point of $T$.  

Sheaf means coherent sheaf unless we  state the contrary. Similarly line/vector bundle  means algebraic (or holomorphic) line/vector bundle   unless we  state the contrary. Abusing notation we  line/vector bundle are considered synonyms of invertible/locally-free sheaf.

Let $\cE,\cF$ be  sheaves on a projective variety $X$. We let
\begin{equation}
\chi_X(\cE,\cF)\coloneqq
\sum_p (-1)^p\dim\Ext^p_X(\cE,\cF).
\end{equation}

\section{Moduli of semistable sheaves  on $K3$ surfaces}\label{sec:kappatre}
\subsection{Deformations of sheaves}\label{subsec:deformofasci}
\setcounter{equation}{0}
Let $X$ be an irreducible smooth projective variety. A \emph{family of  sheaves on $X$  parametrized by a scheme $T$}
consists of a $T$-flat sheaf   $\sF$ on $X\times T$. If $t\in T$ we let  $\sF_t\coloneqq\sF_{|X\times\{t\}}$. We think of $\sF$ as a \lq\lq continuous\rq\rq\ family of sheaves $\sF_t$ (with \lq\lq non reduced\rq\rq\ deformations if $T$ is non reduced).
Let $\cF$ be a  sheaf on $X$. Let $T$ be a pointed scheme, i.e.~a scheme with a  chosen  $0\in T$. A \emph{deformation} of 
$\cF$ parametrized by $T$ is a family of  sheaves on $X$  parametrized by $T$ together with an isomorphism 
$\sF_0\xrightarrow{\sim}\cF$. 
Now suppose that $\cF$ is simple, i.e.~that $\Hom_X(\cF,\cF)=\CC\Id_{\cF}$. Then there exist a scheme $B$ and a deformation of $\cF$  parametrized by $B$   with the following properties (see~\cite{artamkin:defdifasci,mukai:lisciosimpl} and~\cite[2.A.5]{huy-lehn:librofasci}):
\begin{enumerate}
\item
Given   a deformation  $\sG$ of $\cF$ parametrized by $T$
there exist a morphism of germs 
$f\colon (T,0)\to (B,0)$  and  an isomorphism 
$\phi\colon f^{*}\sF\xrightarrow{\sim}\sG_{|X\times (B,0)}$ (here $ f^{*}\sF$ is shorthand for $(\Id_X\times f)^{*}\sF$).
\item
Given $\sG$ as above, there is a unique morphism $f$ with the  property stated above. 
\end{enumerate}
The properties above determine 
the algebraic germ  
$(B,0)$ and the deformation $\sG_{|X\times (B,0)}$ of $\cF$ up to isomorphism: this is the \emph{universal deformation space} of $\cF$. We denote 
$(B,0)$ by $\Def(\cF)$.
One 
describes  
$(B,0)$ as follows. Let
\begin{equation}
\Ext^p_X(\cF,\cF) \xrightarrow{\Tr^p}H^p(X,\cO_X)
\end{equation}
be the \emph{trace map} (if $\cF$ is locally-free $\Tr^p$ is the map $H^p(X,End\cF) \to H^p(X,\cO_X)$ induced in cohomology by the trace morphism 
$Tr\colon End\cF\to\cO_X$), and let $\Ext^p_X(\cF,\cF)_0$ be the kernel of $\Tr^p$. There is an analytic \emph{obstruction map}
\begin{equation}\label{mappaostr}
(\Ext^1_X(\cF,\cF),0) \xrightarrow{\obs_{\cF}} (\Ext^2_X(\cF,\cF)_0,0)
\end{equation}
 with vanishing differential such that the \emph{analytic germ} $(B,0)_{\an}$ is given by
\begin{equation}
(B,0)_{\an}\cong (\obs_{\cF}^{-1}(0),0).
\end{equation}
Since $\obs_{\cF}$ has vanishing differential, it follows that the Zariski tangent space to $B$ at $0$ is identified with $\Ext^1_X(\cF,\cF)$. We also get that  if $\Ext^2_X(\cF,\cF)_0$ vanishes then $\cF$ is unobstructed, i.e.~the deformation space is smooth of dimension equal to 
the dimension of $\Ext^1_X(\cF,\cF)$. 
\begin{thm}[Mukai-Artamkin]\label{thm:mukart}
Let $S$ be  a projective $K3$ surface, and let
$\cF$ be a simple torsion-free sheaf on $S$. Then the deformation space $\Def(\cF)$ is smooth, and its dimension is given by
\begin{equation}\label{dimdefk3}
\dim\Def(\cF)=2-\chi_S(\cF,\cF).
\end{equation}
\end{thm}
\begin{proof}
Since the canonical bundle of $S$ is trivial, Serre duality gives the isomorphism $\Ext^2_S(\cF,\cF)\cong\Hom_S(\cF,\cF)^{\vee}$ and the latter is 
one-dimensional because $\cF$ is simple. Thus
\begin{equation}\label{entrambiuno}
\dim \Ext^2_S(\cF,\cF)=\dim\Hom_S(\cF,\cF)=1.
\end{equation}
The map  $H^0(S,K_S)=H^0(S,\cO_S)\to \Hom_S(\cF,\cF)$ defined by $f\mapsto f\Id_{\cF}$ is identified by Serre duality with the transpose
 of $\Tr^2$, and it is an isomorphism because $\cF$ is simple and torsion-free. It follows that  
\begin{equation}\label{oboemarcello}
 \Ext^2_S(\cF,\cF)_0=0.
\end{equation}
Hence 
 $\Def(\cF)$ is smooth of dimension equal to $\dim\Ext^1_S(\cF,\cF)$. By~\eqref{entrambiuno} the latter is equal to the right hand side of~\eqref{dimdefk3}.
\end{proof}
Let $\cF$ be a sheaf on an irreducible smooth projective variety $X$. We need to consider also deformations of the couple $(X,\cF)$ in the analytic category. Such a deformation consists of a deformation of $X$, i.e.~a proper flat map $f\colon \cX\to T$, where $T$ is a pointed analytic space with base point $0\in T$, and an isomorphism 
$f^{-1}(0)\eqqcolon X_0\xrightarrow{\sim}X$, together with a $T$-flat sheaf $\sF$ on $\cX$ with an isomorphism 
$\sF_{|X_0}\xrightarrow{\sim}\cF$. We wish to compare such deformations with deformations of $X$. For simplicity we assume that $X$ has no non-zero global holomorphic vector fields, and hence it has a universal deformation space $\Def(X)$ (see for example~\cite{manetti-libro}). We recall that 
$\Def(X)$ is an analytic germ with Zariski tangent space $H^1(X,\Theta_X)$, and that if $X$ is K\"ahler with trivial canonical bundle then it is smooth by the Bogomolov-Tian-Todorov Theorem.  Suppose also that $\cF$ is simple. Then there is a universal deformation space $\Def(X,\cF)$ for deformations of the couple $(X,\cF)$ (see~\cite{iacono-manetti:defcoppie}), equipped with the forgetful map
\begin{equation}\label{ricordosolodet}
\Def(X,\cF) \lra \Def(X,\det\cF).
\end{equation}
\begin{thm}[\cite{huang:modcoppie,iacono-manetti:defcoppie}]\label{thm:defcoppia}
Let $(X,\cF)$ be a couple consisting  either of a smooth projective variety $X$  and 
  a simple  sheaf $\cF$  on $X$, or a compact complex manifold $X$ and a simple locally-free sheaf $\cF$  on $X$.  Suppose that the trace map $\Tr^1$ is surjective and the trace map $\Tr^2$ is injective. Then the forgetful map in~\eqref{ricordosolodet}  is smooth. 
\end{thm}
\begin{crl}\label{crl:defcoppia}
Let $S$ be a $K3$ surface and let $\cF$ be a simple torsion-free sheaf on $S$. The forgetful map in~\eqref{ricordosolodet} (with $X=S$) is smooth. 
\end{crl}
\begin{proof}
The trace map $\Tr^1$ is surjective because $H^1(S,\cO_S)=0$. In the proof of Theorem~\ref{thm:mukart} we have  shown that $\Tr^2$ is an isomorphism. Hence the corollary follows from Theorem~\ref{thm:defcoppia}.
\end{proof}
The content of the above result is that we can extend the sheaf $\cF$ to all (small) deformations of $S$ that keep the first Chern class of $\cF$ of type 
$(1,1)$. Note that  $\Def(S,\det\cF)=\Def(S)$ only if  $\det\cF$ is trivial. In all other cases 
$\Def(S,\det\cF)$ has codimension $1$ in $\Def(S)$, i.e.~there is \lq\lq one deformation direction of $S$\rq\rq\ in which $\cF$ does not extend. If $\cF$ is locally-free  then the  projectivization  $\PP(\cF)$ extends to all (small) deformations of $S$. In order to put  this result in context we recall a general result. Let $X$ be a connected complex manifold, and let $f\colon Y\to X$ be a fibration in $m$-dimensional projective spaces, i.e.~locally $f$ is given by the projection $U\times\PP^m\to U$, where \lq\lq locally\rq\rq\ means in the classical topology. We have the forgetful 
map (see~\cite[Expl.~8.2.14]{manetti-libro})
\begin{equation}\label{fibratobase}
\Def(Y) \lra \Def(X).
\end{equation}
(Every deformation of $Y$ is a deformation of the fibration $f\colon Y\to X$.)
If $X$ is a complex manifold  and $\cF$ is a vector bundle on $X$ we let $End^0(\cF)\subset End(\cF)$ be the sub vector bundle of traceless endomorphisms. 
Note that $\Ext^2(\cF,\cF)_0$ is isomorphic to $H^2(X,End^0(\cF))$.
\begin{thm}[Horikawa]\label{thm:evvivahorikawa}
Let $X$ be a connected complex manifold with no nonzero global (holomorphic) tangent field. Let $\cF$ be a vector bundle on $X$ 
such that $H^2(X,End^0(\cF))=0$. Then the  forgetful  map $\Def(\PP(\cF)) \lra \Def(X)$ is 
 smooth.
\end{thm}
\begin{proof}
Let $Y\coloneqq \PP(\cF)$, and let $f\colon Y\to X$ be the structure map. We have  the exact sequence of locally free sheaves on $X$
\begin{equation}\label{succdiff}
0\lra \Theta_{Y/X}\lra \Theta_Y\overset{df}{\lra} f^{*}\Theta_X\lra 0.
\end{equation}
By~\cite[Theorem 6.1]{horikawa-II} (see also~\cite[Cor.~8.2.14]{manetti-libro}) it suffices to prove that 
the map  $H^1(Y,\Theta_Y)\to H^1(Y,f^{*}\Theta_X)$ is surjective and that the map  $H^2(Y,\Theta_Y)\to H^2(Y,f^{*}\Theta_X)$ is injective. By the exact sequence in~\eqref{succdiff} it suffices to show that $H^2(Y,\Theta_{Y/X})=0$. 
By the Leray spectral sequence abutting to $H^2(Y,\Theta_{Y/X})$ we are done if we prove that 
\begin{equation}\label{nemequittepas}
H^p(Y,R^q f_{*}(\Theta_{Y/X}))=0\qquad p+q=2.
\end{equation}
We have
\begin{equation}
R^q f_{*}(\Theta_{Y/X}))\cong 
\begin{cases}
End^0 \cF & \text{if $q=0$,} \\
0 & \text{if $q>0$.} 
\end{cases}
\end{equation}
Thus~\eqref{nemequittepas} holds by hypothesis.
\end{proof}
\begin{crl}\label{crl:martysupreme}
Let $S$ be a $K3$ surface and let $\cF$ be a simple vector bundle on $S$. The forgetful 
map $\Def(\PP(\cF)) \lra \Def(S)$ is smooth.
\end{crl}
\subsection{Stability and semistability}
\setcounter{equation}{0}
\subsubsection{Gieseker-Maruyama (semi)stability, and moduli spaces}
Let $(X,L)$ be a polarized irreducible smooth projective variety, i.e.~$X$ is an irreducible  smooth projective variety and $L$ is an ample line bundle on $X$. We always assume that $c_1(L)$ is primitive, i.e.~$c_1(L)=d\alpha$ with $\alpha\in H^2(X,\ZZ)$ implies that $d=\pm 1$. Let $\cF$ be a non zero torsion-free sheaf on $X$: the \emph{reduced Hilbert polynomial} 
$p_\cF\in\QQ[t]$ is defined by setting
\begin{equation}\label{polhilbrid}
p_\cF(n)\coloneqq\frac{\chi(X,\cF\otimes L^{\otimes n})}{r(\cF)},
\end{equation}
where $r(\cF)$ is the rank of $\cF$ (which is non zero because $\cF$ is non zero and torsion-free).
\begin{dfn}
A non zero torsion-free sheaf $\cF$ on $X$ is \emph{$L$ Gieseker-Maruyama  semistable} if for every non zero subsheaf 
$\cE\subset \cF$ 
\begin{equation}
p_\cE(n)\le p_\cF(n)\qquad \forall n\gg 0.
\end{equation}
It is \emph{$L$ Gieseker-Maruyama stable} if, in addition, strict inequality holds whenever $\cE\not=\cF$. It is \emph{$L$ Gieseker-Maruyama  unstable} if it is not 
$L$ Gieseker-Maruyama semistable.
\end{dfn}
We write  $L$ GM (semi)stable instead of $L$ Gieseker-Maruyama  semistable, and if $L$ is clear from the context we avoid mentioning it. The importance of GM semistability is attested by the following celebrated result.
\begin{thm}[Gieseker~\cite{gieseker:modfasci} - Maruyama~\cite{maruyama:modfasci}]\label{thm:modfasci}
Let $(X,L)$ be a polarized irreducible smooth projective variety, and let
\begin{equation}
v=(v_0,v_1,\ldots)\in\bigoplus _p H^{p,p}_{\QQ}(X)
\end{equation}
with $v_0>0$. 
There is a projective \emph{coarse moduli space} scheme $M_v(X,L)$ for torsion-free sheaves $\cF$ on $X$ such that 
$\ch(\cF)=v$. This means the following:
\begin{enumerate}
\item
To every  scheme $T$ and a family $\sF$ of  sheaves on $X$  parametrized by $T$
 such that  for $t\in T$ the sheaf $\sF_t\coloneqq\sF_{|X\times\{t\}}$ is semistable with $\ch(\sF_t)=v$ (we say that $\sF$ is a \emph{family of semistable sheaves on $X$ with Chern character $v$ parametrized by $T$}) there is associated a  classifying morphism $m_{\sF}\colon T\to M_v(X,L)$ in such a way that the following hold:
\begin{enumerate}
\item
If $\sF'$ is obtained from $\sF$ via a base change $g\colon T'\to T$, i.e.~$\sF'\cong g^{*}(\sF)$, then $m_{\sF'}=m_{\sF}\circ g$.
\item
If $\cL$ is a line bundle on $T$ then $m_{\sF\boxtimes\cL}=m_{\sF}$.
\item
$M_v(X,L)$ is maximal with the above properties, i.e.~if $\cM$ is a scheme with a contravariant functor $F$ from the category of families... as above, then $F$ factors via a morphism $M_v(X,L)\to\cM$.
\end{enumerate}
\item
Restricting the classifying morphism to stable sheaves $\cF$ on $X=X\times\Spec\CC$ with $\ch(\sF_t)=v$  one gets a bijection between the set of isomorphism classes of such sheaves and the set of closed points of an open subset $M_v(X,L)^{\rm st}_{\rm red}$ of the (reduced) $M_v(X,L)_{\rm red}$. 
\item
Restricting the classifying morphism to semistable non stable sheaves $\cF$ on $X=X\times\Spec\CC$ with 
$\ch(\sF_t)=v$  one gets a bijection between the set of {$S$-equivalence} (a relation 
weaker than isomorphism) classes 
(see~\cite[Sect.~1.5]{huy-lehn:librofasci}) of  such sheaves  and the set of closed points of  
$M_v(X,L)_{\rm red}\setminus M_v(X,L)^{\rm st}_{\rm red}$. 
\end{enumerate}
\end{thm}
\begin{dfn}\label{dfn:puntoesse}
If  $\cF$ is an $L$ semistable torsion-free sheaf on $X$ with $\ch(\cF)=v$ 
we let $[\cF]\in M_v(X,L)$ be the point corresponding to $\cF$ according to Items~(2) and~(3) of Theorem~\ref{thm:modfasci}. 
\end{dfn}
\begin{expl}
Let $\alpha\in H^{1,1}_{\ZZ}(X)$, and let $v(\alpha)=e^{\alpha}$. Then $M_{v(\alpha)}(X,L)=\Pic^{\alpha}(X)$ is independent of $L$, and it is  isomorphic to the Jacobian $H^{0,1}(X)/H^1(X;\ZZ)$. This is the prototoypical example of a moduli space. There exists a universal (Poincar\'e) line bundle $\mathsf L$ on 
$X\times M_{v(\alpha)}(X,L)$, i.e.~such that the associated morphism  $m_{\mathsf L}\colon M_{v(\alpha)}(X,L)\to M_{v(\alpha)}(X,L)$ is the identity. In general there is no  universal sheaf on $X\times M_{v}(X,L)$ (even if all the sheaves parametrized by $M_{v}(X,L)$ happen to be stable). 
\end{expl}
\begin{rmk}
One of the steps in the proof of Theorem~\ref{thm:modfasci} is a boundedness result for torsion-free semistable sheaves on $X$  with the assigned Chern class $v$, i.e.~the existence of a scheme of finite type $T$ and 
a family of semistable sheaves on $X$ with Chern character $v$ parametrized by $T$
 which \lq\lq contains\rq\rq\ all such families of sheaves on $X$, i.e.~such that if  $\sF'$ is a similar family of sheaves on $X$ 
parametrized by $ T'$ 
then there exists a morphism   $T'\to T$ such that  $\sF'$ is isomorphic (up to tensorization by the pull-back of a line bundle on $T'$) to the family obtained from $\sF$  via the base change $T'\to T$.  Maruyama proved that if we limit ourselves to locally-free semistable sheaves, 
i.e.~locally-free sheaves on $X\times T$, then in order 
to obtain boundedness it suffices to fix the first three terms of $v$, i.e.~rank, $\ch_1$ and $\ch_2$, 
see~\cite{maruyama:fascilimitati}.
\end{rmk}
\subsubsection{Slope (semi)stability}
Let $(X,L)$ be a polarized irreducible smooth projective variety.  Let $d\coloneqq \dim X$. Applying the Hirzebruch-Riemann-Roch Theorem, one gets the equality
\begin{equation}\label{polhilbridotto}
p_{\cF}(t)\equiv\left(\int_X\frac{c_1(L)^d}{d!}\right)t^d+
\left(\int_X\left(\frac{c_1(\cF)}{r(\cF)}+c_1(X)\right)\frac{c_1(L)^{d-1}}{(d-1)!}\right)t^{d-1}\pmod{t^{d-2}},
\end{equation}
where $r(\cF)$ is the rank of $\cF$. Hence the leading coefficient of the reduced Hilbert polynomial $p_{\cF}$ is independent of $\cF$, while the next coefficient only depends on the \emph{$L$ slope of} $\cF$, dfined to be
\begin{equation}\label{pendenza}
\mu_{L}(\cF)\coloneqq
\int_X\frac{c_1(\cF) \cdot c_1(L)^{d-1}}{r(\cF)}.
\end{equation}
We denote $\mu_{L}(\cF)$ also by $\mu_{c_1(L)}(\cF)$.
This leads to the following definition.
\begin{dfn}\label{dfn:penstab}
A non zero torsion-free sheaf $\cF$ on $X$ is \emph{$L$ slope semistable} if for every non zero subsheaf 
$\cE\subset \cF$ 
we have $\mu_L(\cE)\le \mu_L(\cF)$.
It is \emph{$L$ slope stable} if, in addition, strict inequality holds whenever $r(\cE)<r(\cF)$.  It is \emph{$L$ slope unstable} if it is not 
$L$ slope semistable.

\end{dfn}
Let  $\cF$ be a non zero torsion-free sheaf on $X$. We have the following chain of implications 
\begin{equation*}
\text{$\cF$  slope stable }\implies \text{$\cF$   GM stable }\implies \text{$\cF$    GM semistable }
\implies \text{$\cF$   slope semistable,}
\end{equation*}
where (semi)stability is with respect to a fixed polarization $L$.
\begin{rmk}\label{rmk:stabcurva}
Let $(X,L)$ be a polarized projective curve (irreducible and smooth). Then $L$ slope (semi)stability is the same as $L$ GM (semi)stability because~\eqref{polhilbridotto} is an equality, and moreover it does not depend on $L$. For this reason  when dealing with a curve we talk of (semi)stability tout court. Note  that if $\cF$ is a vector bundle on $X$, and the numbers $r(\cF)$, $\deg(\cF)\coloneqq c_1(\cF)$ are coprime, then $\cF$ is either stable or unstable. 
\end{rmk}
\begin{expl}
Let $X$ be a smooth (irreducible) projctive curve. If $\cF$ is a vector bundle on $X$ fitting into a non split exact sequence
\begin{equation}
0\lra \cO_X\lra \cF\lra\cE\lra 0,
\end{equation}
where $\cE$ is a stable vector bundle with $\deg(\cE)=1$, then $\cF$ is stable (easy exercise). If the genus $g(X)$ is at least $1$ then such non split exact sequences exist with $\cE$ of arbitrary positive rank (induction on the rank). 
If  $g(X)=0$  then there are no stable vector bundles of rank greater than $1$.
\end{expl}
\begin{rmk}\label{rmk:stabaperto}
If  $T$ is a  scheme  and  $\sF$ is family of  sheaves on $X$  parametrized by $T$ then the subsets  $T^{\sss},T^{\st}\subset T$ of 
  $t\in T$ such that the sheaf $\sF_t$ is semistable (respectively stable) is open,  
 see~\cite[Prop.~2.3.1]{huy-lehn:librofasci}. 
\end{rmk}
The result below (for a proof see~\cite[Cor.1.2.8]{huy-lehn:librofasci}) gives a link between stability and simplicity of a sheaf .
\begin{prp}\label{prp:stabsemp}
Let $(X,L)$ be  a polarized irreducible smooth projective variety. If $\cF$ is an $L$ stable sheaf on $X$ then $\cF$ is simple.
\end{prp}
 The  result below makes the connection between the deformation space of a stable sheaf 
 $\cF$ and the local structure at $[\cF]$ of the relevant moduli space.    (For a proof 
 see~\cite[Thm.4.5.1]{huy-lehn:librofasci}.)
\begin{prp}\label{prp:modedef}
Let $(X,L)$ be  a polarized irreducible smooth projective variety. Let $\cF$ be an $L$ stable sheaf on $X$, and let $v=\ch(\cF)$. Then the germ of $M_v(X,L)$ at the point $[\cF]$ representing the isomorphism class of $\cF$  is isomorphic to the universal deformation space of $\cF$.
\end{prp}
\subsection{Variation of slope (semi)stability}\label{subsec:varstab}
\setcounter{equation}{0}
\subsubsection{Motivation and overview}
Let $X$ be an irreducible smooth projective variety. (Semi)stability of a (torsion-free) sheaf with respect to an ample line bundle $L$  depends  only on the ray $\QQ_{+} c_1(L)\subset \Amp(X)$. If $X$ is a curve there is no dependency (see Remark~\ref{rmk:stabcurva}). Starting with $\dim X=2$ slope stability  does vary when we vary  the ray in the ample cone.
\begin{expl}
Let $X$ be a(n irreducible smooth projective) surface. Let $\cF$ be a rank $2$ vector bundle fitting into a non split exact sequence
\begin{equation}\label{fascioambiguo}
0\lra \cO_X\lra \cF\overset{g}{\lra} \cI_Z\otimes \cL\lra 0,
\end{equation}
where $Z\subset X$ is a zero-dimensional subscheme and $\cL$ is a line bundle. Suppose that $h_1,h_2$ be ample classes on $X$  such that
\begin{equation}\label{epifania}
 c_1(\cL)\cdot h_1=1,\qquad  c_1(\cL)\cdot h_2<0.
\end{equation}
 Then $\cF$ is $h_1$ slope stable and $h_2$ slope unstable. The latter statement holds because by the inequality in~\eqref{epifania} we have 
 $\mu_{h_2}(\cO_X)=0>\mu_{h_2}(\cF)$. Let us prove that  $\cF$ is $h_1$ slope stable. Suppose the contrary, i.e.~there exists a subsheaf $\cE\subset\cF$ with $r(\cE)=1$ and $\mu_{h_1}(\cE)\ge \mu_{h_1}(\cF)$. Since $\cF$ is locally-free the inclusion $\cE\subset\cF$ factors through an inclusion 
 $\cE^{\vee\vee}\subset\cF$. Since $\mu_{h_1}(\cE)=\mu_{h_1}(\cE^{\vee\vee})$ and $\cE^{\vee\vee}$ is a line bundle we may assume 
from the start  that $\cE$ is a line bundle. Since $\mu_{h_1}(\cE)\ge \mu_{h_1}(\cF)=1/2$ the sheaf $\cE$ is not a subsheaf of 
 $\cO_X$ It follows that
the sheaf 
$g(\cE)\subset \cI_Z\otimes \cL$ has rank $1$.  Since we have a non zero map $\cE\to\cL$ between line bundles,  $c_1(\cE)=c_1(\cL)-\cl(D)$ where $D$ is an effective divisor. The inequality $\mu_{h_1}(\cE)\ge  \mu_{h_1}(\cF)$ gives
$D\cdot h_1\le 0$, and hence $D=0$. Thus $\cE\cong\cL$ (it follows that 
 $Z=\es$) and  hence  the exact sequence in~\eqref{fascioambiguo} is 
split. This contradicts the hypothesis.

An explicit instance of the above construction is the following. Let $(E_1,p_1)$, $(E_2,p_2)$ be  elliptic curves. Let $X\coloneqq E_1\times E_2$, and let $C_1\coloneqq E_1\times\{p_2\}$, 
$C_2\coloneqq \{p_1\}\times E_2$. If $Z\subset X$ is a l.c.i.~subscheme (e.g.~$Z=\es$) 
 there exists a vector bundle $\cF$ fitting into a non split exact sequence
\begin{equation}\label{magicest}
0\lra\cO_X\lra\cF\lra \cI_Z\otimes\cO_X(C_1-C_2)\lra 0.
\end{equation}
The polarizations $h_1\coloneqq \cO_S(C_1+2C_2)$ and $h_2\coloneqq \cO_S(2C_1+C_2)$ satisfy the hypotheses above. Hence $\cF$ is $h_1$ slope stable and $h_2$ slope unstable.
\end{expl}
 If $X$ is a surface the variation of (semi)stability is very well behaved, see~\cite{huy-lehn:librofasci} and  references therein. If $\dim X>2$  variation of stability is, in general, substantially more complex, see~\cite{grebrosstoma:fibvettaltadim}. 
   In the present subsection we discuss variation  of slope (semi)stability for 
 an irreducible smooth projective surface, and $X$ denotes such a variety   unless we state the contrary. The argument is technical and dry, but it is a key element in many proofs on moduli of sheaves.
\begin{rmk}
 Bridgeland introduced stability conditions for objects in the derived category of a smooth projective variety. In a (precise) sense this is  a generalization of Gieseker-Maruyama stability for sheaves. 
 Such stability conditions, when  defined  and controllable, have been extremely fruitful, see~\cite{bayermacri:icm} for a survey.  Variation of Bridgeland stability,  a high octane version of variation of slope stability for sheaves on surfaces,    plays a key r\^ole in the theory. 
\end{rmk}
\subsubsection{Slope with respect to nef classes}
 It is convenient to introduce  slope (semi)stability of a sheaf $\cF$ on a surface $X$ with respect to a nef class  $h\in\NS(X)_{\QQ}$.  The \emph{$h$ slope of} 
$\cF$ is given by
\begin{equation}
\mu_{h}(\cF)\coloneqq
\int_X\frac{c_1(\cF) \cdot h}{r(\cF)}.
\end{equation}
\begin{dfn}
Let $X$ be a smooth projective surface, and let $h\in\NS(X)_{\QQ}$ be a nef class.  A non zero  torsion-free sheaf $\cF$ on $X$ is \emph{$h$ slope semistable} if for every non zero subsheaf 
$\cE\subset \cF$ 
we have $\mu_h(\cE)\le \mu_h(\cF)$.
It is \emph{$h$ slope stable} if, in addition, strict inequality holds whenever $r(\cE)<r(\cF)$. It is \emph{$h$ slope unstable} if it is not 
$h$ slope semistable.
\end{dfn}
Of course if $h$ is ample this is the same notion of slope (semi)stability of Definition~\ref{dfn:penstab}. 
\begin{rmk}\label{rmk:stabnonampio}
Suppose that $\pi\colon X\to Y$ is a surjective non constant map to a smooth projective variety, and  that $\pi_{*}\cO_X=\cO_Y$. Let  $h_Y$ be an ample class on $Y$, and let 
$h_X=\pi^{*}h_Y$. There are two  possible cases:
\begin{enumerate}
\item[(a)]
$\dim Y=2$, and hence $\pi$ is a composition of blow-downs.
\item[(b)]
$\dim Y=1$, and hence a general fiber of $\pi$ is an irreducible smooth curve. 
\end{enumerate}
If~(a) holds then a sheaf $\cF$ is $h_X$ slope (semi)stable if and only if the push-forward $\pi_{*}\cF$ is $h_Y$ slope (semi)stable. This holds because of the following facts: $\mu_{h_X}(\cF)=\mu_{h_Y}(\pi_{*}\cF)=\mu_{h_X}(\pi^{*}\pi_{*}\cF)$, 
if  $\cE\subset\cF$ then $\pi_{*}\cE\subset\pi_{*}\cF$,  if  $\cG\subset\pi_{*}\cF$ then 
$\pi^{*}\cG\subset\pi^{*}\pi_{*}\cF\subset\cF$ and $\mu_{h_Y}(\cG)=\mu_{h_X}(\pi^{*}\cG)$.
If~(b) holds, we may assume that $h_Y$ is the class of a point, and hence $h_X$ is the class of a fiber of $\pi$.
 Let  $\cE$ be a (non zero) torsion-free sheaf on $X$. If $C$ is a general fiber of $\pi$ then the restriction $\cE_{|C}$ is a vector bundle, and
\begin{equation}
\mu_{h_X}(\cE)=\int_X\frac{c_1(\cE)\cdot C}{r(\cE)}=\int_C\frac{\deg(\cE_{|C})}{r(\cE)}=\mu(\cE_{|C}).
\end{equation}
 It follows that if  $\cF$ is a torsion-free sheaf on $X$ whose restriction to a general fiber of $\pi$ is a slope (semi)stable,  then $\cF$ is $h_X$ slope (semi)stable. 
\end{rmk}
\subsubsection{Suitability}
The \emph{discriminant} of a   torsion-free sheaf $\cF$ on a connected complex manifold $X$ (of any dimension) is given by
\begin{equation}\label{eccodisc}
 \Delta(\cF)\coloneqq 2r(\cF) c_2(\cF)-(r(\cF)-1) c_1(\cF)^2=\ch_1(\cF)^2-2r(\cF)\ch_2(\cF).
\end{equation}
Note that if $\cL$ is a line bundle on $X$ then $\Delta(\cF\otimes\cL)=\Delta(\cF)$. The celebrated Bogomolov inequality for the discriminant  underpins the results in the present subsection. 
\begin{thm}[Bogomolov's inequality.~\cite{bogomolov:discpos,bogomolov:reststab,huy-lehn:librofasci}]\label{thm:disegbog}
Let $(X,h)$ be an irreducible  smooth polarized projective variety of dimension $d\ge 2$. If 
$\cF$ is a torsion-free  $h$ slope semistable sheaf on $X$ then
\begin{equation}\label{bogdiseg}
\int_X \Delta(\cF)\cdot h^{d-2}\ge 0.
\end{equation}
If $d=2$ then~\eqref{bogdiseg} holds also if  $h$ is big and nef.
\end{thm}
Suppose that $X$ is a surface.  The orientation class of $X$ defines an isomorphism $H^4(X;\ZZ)\xrightarrow{\sim} \ZZ$, and hence 
it identifies $\Delta(\cF)$  with an integer.  
 We let
\begin{equation}\label{esmeralda}
{\mathsf a}(\cF)\coloneqq r(\cF)^2 \cdot \Delta(\cF)/4.
\end{equation}  
\begin{prp}\label{prp:lizzy}
Let $X$ be a smooth projective surface, and let  $h\in\NS(X)$ be big and nef. Let $\cF$ be an $h$ slope semistable   torsion-free sheaf  on $X$. Suppose that $\cE\subset\cF$ is a subsheaf such that $0<r(\cE)<r(\cF)$ and $\mu_h(\cE)=\mu_h(\cF)$. Let 
\begin{equation}\label{eccolambda}
\lambda_{\cE,\cF}\coloneqq r(\cF)\cdot c_1(\cE)-r(\cE)\cdot c_1(\cF).
\end{equation}
Then 
\begin{equation}\label{lizzy}
-{\mathsf a}(\cF)\le \lambda_{\cE,\cF}\cdot \lambda_{\cE,\cF}\le 0,
\end{equation}
and the right inequality is an equality  only if $\lambda_{\cE,\cF}=0$.
\end{prp}
\begin{proof}
The equality  $\mu_h(\cE)=\mu_h(\cF)$ is equivalent to  $\lambda_{\cE,\cF}\cdot h=0$. Hence the Hodge index Theorem gives that 
the right inequality in~\eqref{lizzy} holds, with equality only if $\lambda_{\cE,\cF}=0$. Since  $\mu_h(\cE)=\mu_h(\cF)$ and $\cF$ is $h$ slope semistable, so are $\cE$ and the quotient $\cF/\cE$ (we may assume that $\cF/\cE$ is torsion-free). Applying Bogomolov's inequality to the semistable sheaves $\cE$ and  $\cF/\cE$ one gets 
 the left inequality in~\eqref{lizzy}, see~\cite[Thm.~4.C.3, Thm.~7.3.3]{huy-lehn:librofasci}.
\end{proof}
\begin{dfn}\label{dfn:poladatta}
Let $X$ be a smooth projective surface, and let  $h_0\in\NS(X)_{\QQ}$ be  nef. 
Let  ${\mathsf a}> 0$. An ample class $h$ is \emph{${\mathsf a}$-suitable for $h_0$}  if  for all
 $\lambda\in \NS(X)$  such that $ -{\mathsf a}\le \lambda\cdot\lambda<0$ one of the following holds:
\begin{enumerate}
\item
$\lambda\cdot h>0$ and $\lambda\cdot h_0\ge 0$.
\item
 $\lambda\cdot h=0$  and $\lambda\cdot h_0= 0$.
\item
 $\lambda\cdot h<0$  and $\lambda\cdot h_0\le 0$.
\end{enumerate}
\end{dfn}
\begin{rmk}
The reason for giving the above definition is that it allows to compare $h_0$ slope semistability and $h$ slope semistability for torsion-free sheaves such that ${\mathsf a}(\cF)\le {\mathsf a}$, see Proposition~\ref{prp:paragonestab} and Subsubsections~\ref{subsubsec:muricamere}, \ref{subsubsec:stabsupfibr}. 
\end{rmk}
\begin{rmk}\label{rmk:fratelli}
   Let $X$ be a smooth projective surface, and let ${\mathsf a}> 0$. 
Let $h_0,h\in\NS(X)_{\QQ}$ with $h_0$ nef and $h$ ample. Thus $Nh_0+h$ is ample for any positive integer $N$. 
If $N$ is large enough  then $Nh_0+h$ is ${\mathsf a}$-suitable for $h_0$.  If $h_0\cdot h_0>0$ this follows from the wall-and-chamber decomposition of the ample cone 
 discussed in Subsubsection~\ref{subsubsec:muricamere}. If $h_0\cdot h_0=0$ see~\cite[Lemma~2.3]{friedman-rk2-ell}.
\end{rmk}
\begin{prp}\label{prp:paragonestab}
Let $X$ be a smooth projective surface, and let  $h_0\in\NS(X)_{\QQ}$ be  nef. 
Let  ${\mathsf a}> 0$, and let $h\in\NS(X)_{\QQ}$  be  ample and ${\mathsf a}$-suitable for $h_0$. 
 Suppose that  $\cF$ is  a torsion-free  sheaf on $X$ such that  
   ${\mathsf a}(\cF)\le {\mathsf a}$, and   which  is
  $h_0$ slope semistable  but not $h$ slope stable.  
    Then there exists a subsheaf  $\cE\subset \cF$ such that 
\begin{equation}
0<r(\cE)<r(\cF),\qquad \lambda_{\cE,\cF}\cdot h\ge 0,\qquad \lambda_{\cE,\cF}\cdot h_0= 0.
\end{equation}
\end{prp}
\begin{proof}
If $h_0\cdot h_0>0$ see~\cite[Thm~4.C.3]{huy-lehn:librofasci}. If $h_0\cdot h_0=0$ the proof is analogous to that of~\cite[Thm~3.3]{friedman-so(3)-invrnts} (one also gets a proof by rewriting  the proof of~\cite[Prop.~5.9]{og:highdim} with the intersection form of the surface 
 replacing the Beauville-Bogomolov-Fujiki symmetric bilinear form). 
\end{proof}
\subsubsection{Wall-and-chamber decomposition of the ample cone}\label{subsubsec:muricamere}
Let ${\mathsf a}$ be  a positive real number. A codimension-$1$ real subspace $W\subset\NS(X)_{\RR}$ is an ${\mathsf a}$-wall if $W=\xi^{\bot}$  where $\xi\in\NS(X)$ is such that
$-{\mathsf a}\le \xi\cdot\xi<0$. 
The intersections of the  ${\mathsf a}$-walls with the ample cone $\Amp(X)$  form a locally finite collection of codimension-$1$ submanifolds (see~\cite[Lemma~4.C.2]{huy-lehn:librofasci}), and hence their union is closed in 
$\Amp(X)$.  A connected component of 
\begin{equation}
\Amp(X)\setminus\bigcup_{\stackrel{\xi\in\NS(X)}{ -{\mathsf a}\le q_X(\xi)<0}}\xi^{\bot}
\end{equation}
 is an 
\emph{open ${\mathsf a}$-chamber}.  An ample class $h\in\NS(X)_{\QQ}$ is \emph{${\mathsf a}$-generic} if it does not belong to any ${\mathsf a}$-wall, i.e.~it belongs to a (unique) open ${\mathsf a}$-chamber. 
\begin{prp}\label{prp:padovagoa}
Let $X$ be an irreducible smooth projective surface. Let $h_0,h_1\in\NS(X)_{\QQ}$ be ${\mathsf a}$-generic ample classes which belong to the same 
open ${\mathsf a}$-chamber. Suppose that $\cF$ is a torsion-free sheaf on $X$ such that ${\mathsf a}(\cF)\le {\mathsf a}$. Then the following hold:
\begin{enumerate}
\item
If $\cF$ is $h_0$ slope stable then  it is $h_1$ slope stable, and viceversa.
\item
If $\cF$ is $h_i$ slope semistable ($i\in\{0,1\}$)  and $\cE\subset\cF$ is 
a subsheaf  with $0<r(\cE)<r(\cF)$ and $\mu_{h_i}(\cE)=\mu_{h_i}(\cF)$, then
$r(\cF)c_1(\cE)-r(\cE)c_1(\cF)=0$. 
\item
If $\cF$ is $h_0$ slope semistable but non stable then it is $h_1$ slope semistable but non stable, and viceversa.
\end{enumerate}
\end{prp}
\begin{proof}
Since $h_0,h_1$ are ${\mathsf a}$-generic  and belong to the same 
open ${\mathsf a}$-chamber, 
 $h_0$ is ${\mathsf a}$-suitable for $h_1$ and viceversa. Let us prove Item~(1) Suppose that $\cF$ is $h_0$ slope stable. If $\cF$ is not $h_1$ slope stable then by Proposition~\ref{prp:paragonestab} (with $h=h_1$)  there exists a subsheaf  $\cE\subset \cF$ such that 
$0<r(\cE)<r(\cF)$ and $\lambda_{\cE,\cF}\cdot h_0= 0$. This is a contradiction because $\lambda_{\cE,\cF}\cdot h_0= 0$ is equivalent to 
$\mu_{h_0}(\cE)=mu_{h_0}(\cF)$. Since the r\^oles of $h_0,h_1$ are symmetric this finishes the proof of Item~(1). We prove Item~(2). 
Since $\mu_{h_i}(\cE)=\mu_{h_i}(\cF)$ we have $\lambda_{\cE,\cF}\cdot h_i= 0$. 
By Proposition~\ref{prp:lizzy} and the hypothesis that $h_i$ does not belong to any ${\mathsf a}$-wall it follows that 
$\lambda_{\cE,\cF}= 0$. Lastly we prove Item~(3). Suppose that $\cF$ is $h_0$ slope semistable but non stable. By Item~(2) there exists a chain
of sheaves $0\not=\cE_1\subsetneq\cE_2\subsetneq\ldots\subsetneq\cE_m=\cF$ such that $\cE_i/\cE_{i-1}$ is $h_0$ slope stable and 
$r(\cF)c_1(\cE_i/\cE_{i-1})-r(\cE_i/\cE_{i-1})c_1(\cF)=0$ (this a \emph{slope Jordan-H\"older filtration} of $\cF$). One has 
${\mathsf a}(\cE_i/\cE_{i-1})<{\mathsf a}(\cF)\le {\mathsf a}$, and hence by Item~(1) each $\cE_i/\cE_{i-1}$ is $h_1$ slope stable. 
Since $r(\cF)c_1(\cE_i/\cE_{i-1})-r(\cE_i/\cE_{i-1})c_1(\cF)=0$ for all $i$, it follows that $\cF$ is $h_1$ slope semistable but non stable. This finishes the proof of Item~(3) because the r\^oles of $h_0,h_1$ are symmetric. 
\end{proof}
\subsubsection{Stability for surfaces fibered over a curve}\label{subsubsec:stabsupfibr}
Suppose that $X$ is an irreducible smooth projective surface equipped with a fibration $\pi\colon X\to Y$ to a smooth curve, i.e.~$\pi$ is surjective  with 
$\pi_{*}\cO_X=\cO_X$. For $y\in Y$ let $C_y\coloneqq\pi^{-1}(y)$. 
If $y\in Y$  is a general point   then
$C_y$ is a smooth irreducible curve which does not meet the finite subset of $X$ where $\cF$ is not locally-free, and hence 
$\cF_{y}\coloneqq \cF_{|C_y}$ is a vector bundle. By openness of stability (see Remark~\ref{rmk:stabaperto}) it follows that  either $\cF_{y}$ 
 is a stable vector bundle for
a general  $y\in Y$,   or it is a non stable vector bundle
  for general  $y\in Y$. 
\begin{prp}\label{prp:fvsupfib}
Let $\pi\colon X\to Y$ be  a fibration from  an irreducible smooth projective surface to a smooth curve.
 Let $h\in\NS(X)_{\QQ}$ be ample and ${\mathsf a}$-suitable for $\cl(C)$, where $C$ is a fiber of $\pi$.  Let $\cF$ be a torsion-free sheaf on $X$ such that 
 ${\mathsf a}(\cF)\le {\mathsf a}$, and the numbers
 $r(\cF)$, $c_1(\cF)\cdot C$ are coprime. Then one of the following holds:
\begin{enumerate}
\item
For a general $y\in Y$  the restriction $\cF_{y}= \cF_{|C_y}$ is  stable and $\cF$ is $h$ slope stable.
\item
For a general $y\in Y$  the restriction $\cF_{y} \cF_{|C_y}$ is  unstable and $\cF$ is $h$ slope unstable. 
\end{enumerate}
\end{prp}
\begin{proof}
Let  $C_y$ be a fiber of $\pi$. Then
\begin{equation}
r(\cF_y)=r(\cF),\qquad \deg(\cF_y)=c_1(\cF)\cdot C_y.
\end{equation}
Assume that $C_y$ is smooth  and that it does not meet the finite subset of $X$ where $\cF$ is not locally-free. Since  $r(\cF)$ and 
$c_1(\cF)\cdot C_y$ are coprime the restriction $\cF_y$ is either stable or unstable (see Remark~\ref{rmk:stabcurva}). 
Hence if $y\in Y$ is general then
 $\cF_y$    is either  stable or  unstable. It remains to prove that if 
it is stable then  $\cF$ is $h$ slope stable, and that if 
it is unstable  then  $\cF$ is $h$ slope unstable. 

Suppose that  $\cF_y$    is  stable for general $y\in Y$.  
By Remark~\ref{rmk:stabnonampio} $\cF$ is $\cl(C)$ slope stable. Suppose that $\cF$ is not $h$ slope stable. By Proposition~\ref{prp:paragonestab} (with $h_0=\cl(C)$) there  exists a subsheaf  $\cE\subset \cF$ such that 
$0<r(\cE)<r(\cF)$ and $\lambda_{\cE,\cF}\cdot \cl(C)= 0$. If the quotient $\cF/\cE$ is not torsion-free we replace $\cE$ by its saturation (and still call it $\cE$). Then $0<r(\cE)<r(\cF)$ and $\lambda_{\cE,\cF}\cdot \cl(C)\ge 0$. 
If $C_y$ is a general  fiber  the last inequality gives  that
\begin{equation}
\mu(\cE_y)=\mu_{\cl(C)}(\cE)\ge \mu_{\cl(C)}(\cF)=\mu(\cF_y).
\end{equation}
This contradicts the hypothesis. This proves that if  the restriction  $\cF_y$   is  stable for  general  $y\in Y$ then $\cF$ is $h$ slope stable. 

Suppose that the restriction  $\cF_y$    is  unstable for a general  $y\in Y$.  
 One proves
 (see~\cite[Thm.~2.3.2]{huy-lehn:librofasci}) that 
there exists a subsheaf  $\cE\subset \cF$ such that 
$0<r(\cE)<r(\cF)$ and $\mu(\cE_y)>\mu(\cF_y)$ for a general   $y\in Y$. This gives that $\lambda_{\cE,\cF}\cdot \cl(C)>0$. 
By Definition~\ref{dfn:poladatta} it follows that $\lambda_{\cE,\cF}\cdot h>0$, and hence $\cF$ is $h$ slope unstable.
\end{proof}
\begin{expl}\label{expl:rigidisuk3ell}
Let $S$ be a (projective) $K3$ surface with  an elliptic fibration $\pi\colon S\to\PP^1$ with a section $\Sigma$. There exist  vector bundles $\cF^1,\cF^2,\ldots,\cF^r,\ldots$ 
(see~\cite{og:accaduespmod}) with the following properties
\begin{enumerate}
\item
$\cF^1\cong\cO_S(\Sigma)$.
\item
For all $r\in\NN_{+}$ there exists a non split exact sequence
\begin{equation*}
0\lra\cO_S\lra \cF^{r+1}\lra  \cF^{r}\otimes\cO_S(-2C)\lra 0.
\end{equation*}
\item
For all $r\in\NN_{+}$ and all $t\in\PP^1$  the restriction of $\cF^r_{|C_t}$ is stable. 
\end{enumerate}
A computation gives that $\Delta(\cF^r)=2r^2-2$, and hence ${\mathsf a}(\cF^r)=(r^4-r^2)/2$ (see~\eqref{esmeralda}). Let $h\in\NS(S)_{\QQ}$  be 
ample and $(r^4-r^2)/2$-suitable for $\cl(C)$. By 
Proposition~\ref{prp:fvsupfib} we get that $\cF^r$ is $h$ slope stable. Note that $v(\cF^r)^2=-2$.
\end{expl}
\subsection{Moduli of semistable sheaves on $K3$ surfaces}
\setcounter{equation}{0}
\subsubsection{The Mukai lattice and the dimension of  moduli spaces}
In  the present subsection $S$ is a $K3$ surface (not necessarily projective).  If $\cF$ is  a sheaf on $S$ the  \emph{Mukai vector} of $\cF$ is given by
\begin{equation}
v(\cF)\coloneqq \ch(\cF)\cdot\sqrt{\td_S}=\ch(\cF)\cdot(1+\eta_S)\in H(S;\ZZ),
\end{equation}
where $\eta_S$ is the fundamental class of $S$. Suppose that $L$ is an ample line bundle on $S$, and that 
$v\in H^0(S;\ZZ)\oplus H^{1,1}_{\ZZ}(S;\ZZ)\oplus H^{4}(S;\ZZ)$ - we call it a \emph{Mukai vector for $S$}.
We assume that $v$ has positive first entry. Let
\begin{equation}
\cM_v(S,L)\coloneqq M_{v\cdot \td_S^{-1/2}}(S,L),
\end{equation}
i.e.~$\cM_v(S,L)$ is the moduli space of $L$ GM semistable sheaves $\cF$ on $S$ with $v(\cF)=v$. Let 
$\cM_v(S,L)^{\st}\subset \cM_v(S,L)$ be the open (see Remark~\ref{rmk:stabaperto}) subscheme whose closed points parametrize stable sheaves.
In order to discuss the  dimension of $\cM_v(S,L)$ we introduce 
the \emph{Mukai pairing} on the cohomology group $H(S;\ZZ)=\bigoplus_{i=0}^2 H^{2i}(S;\ZZ)$. It  is defined by
\begin{equation}
\la a,b\ra\coloneqq \int_S(a_2\cdot b_2-a_0\cdot b_4-a_4\cdot b_0),
\end{equation}
where $a_{2i},b_{2i}$ are the components of $a,b$ in $H^{2i}(S;\ZZ)$ respectively. Note that the Mukai pairing is even. We use the same notation for the linear extension of the Mukai pairing to the complex  cohomology $H(S)$. If there is no risk of confusion we set $a^2\coloneqq \la a,a\ra$. 

Let $\cE,\cF$ be  sheaves on $S$. The Hirzebruch-Riemann-Roch formula gives that
\begin{equation}\label{caratmukai}
\chi_S(\cE,\cF)=-\la v(\cE),v(\cF)\ra.
\end{equation}
If $\cF$ is a simple sheaf on $S$, then by Serre duality we have $\dim\Ext^2_S(\cF,\cF)=1$. Hence the formula in~\eqref{caratmukai} gives the following result.
\begin{prp}\label{prp:sesempl}
Let  $S$ be a $K3$ surface. If  $\cF$ is a simple   sheaf on $S$ then $ v(\cF)^2 \ge -2$.
\end{prp}
\begin{prp}\label{prp:modstab}
Let  $(S,L)$ be a polarized $K3$ surface. Let $v$ be a Mukai vector for $S$ 
  with positive first entry. Suppose that $\cM_v(S,L)^{\st}$ is non empty. Then $\cM_v(S,L)^{\st}$ is smooth of pure dimension given by
\begin{equation}\label{dimfvstab}
\cM_v(S,L)^{\st}=v^2+2.
\end{equation}
\end{prp}
\begin{proof}
Let $[\cF]$ be a closed point of $\cM_v(S,L)^{\st}$. By Proposition~\ref{prp:modedef} and Theorem~\ref{thm:mukart} $\cM_v(S,L)^{\st}$ is smooth at $[\cF]$, and  by~\eqref{dimdefk3}, \eqref{caratmukai} its dimension at $[\cF]$ is given by the right hand side of~\eqref{dimfvstab}.
\end{proof}
Next we discuss $L$ semistable non stable sheaves  for generic polarizations.
\begin{prp}\label{prp:stabnonstabk3}
Let $S$ be a $K3$ surface. Let $w$ be a primitive 
Mukai vector for $S$ (\lq\lq primitive\rq\rq\  means non divisible by an integer greater than $1$)
 with positive first entry, and let $v=kw$ where $k\in\NN_{+}$. Let 
\begin{equation}\label{adivu}
{\mathsf a}(v)\coloneqq \frac{1}{4}\left(v_0^2\cdot\la v,v\ra+2 v_0^4\right),
\end{equation}
and let $h\in\Amp(S)$ be ${\mathsf a}(v)$-generic. Suppose that $\cF$ is a torsion-free  $h$ semistable sheaf with $v(\cF)=v$, and that there exists a subsheaf 
$\cE\subset\cF$ such that $p_{\cE}=p_{\cF}$ ($p_{\cE},p_{\cF}$ are the reduced Hilbert polynomials, see~\eqref{polhilbrid}). Then 
$v(\cE)=k_1w$ where $k_1$ is an integer and $0<k_1<k$. 
\end{prp}
\begin{proof}
The equality $p_{\cE}=p_{\cF}$ means that $r(\cF) v(\cE)=r(\cE)v(\cF)$. The proposition follows from this and Item~(2) of Proposition~\ref{prp:padovagoa}, because  ${\mathsf a}(\cF)={\mathsf a}(v)$.
\end{proof}
\subsubsection{Symplectic forms}
Let  $(S,L)$ be a polarized $K3$ surface.  Let $v$ be a Mukai vector for $S$ 
 with positive first entry.  Let $\sigma\in H^0(S,K_S)$ be  non zero. 
  One associates to $\sigma$ a holomorphic symplectic form  on  $\cM_v(S,L)^{\st}$ as follows. 
Let $[\cF]$ be a closed point of $\cM_v(S,L)^{\st}$. By  Proposition~\ref{prp:modedef} the germ of $\cM_v(S,L)^{\st}$ at $[\cF]$ is canonically identified with $\Def(\cF),0)$, which has Zariski tangent space at $0$ equal to $\Ext^1_S(\cF,\cF)$ (see 
Subsection~\ref{subsec:deformofasci}). Let
\begin{equation}\label{formasimpl}
\begin{matrix}
\Ext^1_S(\cF,\cF)\times \Ext^1_S(\cF,\cF) & \xrightarrow{\tau_v(\sigma)([\cF])} & \CC \\
(\alpha,\beta) & \mapsto & \int_S \Tr(\alpha\cup\beta)\wedge\sigma
\end{matrix}
\end{equation}
where $\Tr(\alpha\cup\beta)\in H^{0,2}(S)$ is given by the trace of the Yoneda product, which is an element of $H^2(\cO_S)$, identified with
$H^{0,2}(S)$ by the Dolbeault isomorphism and the Hodge decomposition.   
\begin{prp}\label{prp:modsimpl}
Let  $(S,L)$ be a polarized $K3$ surface.  Let $v$ be a Mukai vector for $S$ 
 with positive first entry. Suppose that  $\cM_v(S,L)^{\st}$ is non empty. Let $\sigma\in H^0(S,K_S)$ be  non zero. There is a holomorphic symplectic form  $\tau_v(\sigma)$ on  $\cM_v(S,L)^{\st}$ whose value at 
 a closed point $[\cF]$ of $\cM_v(S,L)^{\st}$ is given by~\eqref{formasimpl}.
\end{prp}
\begin{proof}
The map $\tau_v(\sigma)([\cF])$ is bilinear and skew-symmetric. It is   a perfect pairing by Serre duality. For  a proof that it is the value at $[\cF]$ of a closed holomorphic $2$-form see~\cite{og:santacruz}.
\end{proof}
The result below follows from Propositions~\ref{prp:stabnonstabk3} \ref{prp:modstab}, and~\ref{prp:modsimpl}.
\begin{crl}\label{crl:stabnonstabk3}
Let $S$ be a $K3$ surface, and let $v$ be a primitive 
Mukai vector for $S$. Let $h\in\Amp(S)$ be ${\mathsf a}(v)$-generic  (see~\eqref{adivu}). If $\cM_v(S,h)$ is non empty
then it is smooth of pure dimension $2+\la v,v\ra$, and it carries a (closed) holomorphic symplectic form.
\end{crl}
\subsubsection{Hyperk\"ahler manifolds}\label{subsubsec:accakappa}
For many choices of polarized $K3$ surface $(S,L)$ and Mukai vector $v$ the moduli space $\cM_v(S,L)$ turns out to be a hyperk\"ahler (HK) manifold. 
\begin{dfn}
A connected compact K\"ahler manifold is \emph{hyperk\"ahler} (HK) if it is simply connected and it carries a holomorphic symplectic form spanning the space of holomorphic $2$-forms.
\end{dfn}
HK manifolds are among the building blocks of compact K\"ahler manifolds with vanishing first Chern class, see~\cite{beauv:ciunozero}.
\begin{rmk}\label{rmk:defhk}
Let $X_1$ be a HK manifold. If $X_2$ is a compact K\"ahler manifold which is a deformation of $X_1$ (i.e.~there exist a proper map of complex manifolds $f\colon \cX\to T$ to a connected $T$ and $t_1,t_2\in T$ such that  $f^{-1}(t_i)\cong X_i$), then $X_2$ is also  a HK manifold. 
\end{rmk}
\begin{expl}\label{expl:esempihk}
Two dimensional hyperk\"ahler  manifolds are the same as $K3$ surfaces. We recall that they form a single deformation class (Kodaira). 
Let $m\ge 2$. Beauville~\cite{beauv:ciunozero} gave examples of two distinct deformation classes in each even dimension  $2m$ proceeding as follows. Let $S$ be a smooth projective  surface. We let $S^{[m]}$ be the Hilbert scheme parametrizing subschemes $Z\subset S$ of length $m$.  If $S$ is a projective $K3$ surface then 
$S^{[m]}$ is a HK (projective) manifold  of dimension $2m$. By associating  to a subscheme $Z\subset S$ of length $m$ the ideal sheaf we get an isomorphism $S^{[m]}\xrightarrow{\sim}\cM_v(S,L)$ where $v=(1,0,1-m)$. One has
\begin{equation}
b_2(S^{[m]})=23.
\end{equation}
A HK manifold  is of \emph{Type $K3^{[m]}$} if it is a deformation of $S^{[m]}$. We emphasize that if $X$ is a very general HK manifold  of Type  $K3^{[m]}$ then there is no $K3$ surface $S$ such that $X$ is  isomorphic  (nor birational) to $S^{[m]}$. Here one should remark that $S^{[m]}$ makes sense (Douady) also if $S$ is not a projective $K3$ surface.  
Beauville produced other deformation classes of HK manifolds by generalizing the construction of Kummer surfaces as follows. 
Let $A$ be an abelian surface. The \emph{$m$-th generalized Kummer variety}  associated to $A$ is  the closed subset  $\Kum_m(A)\subset A^{[m+1]}$ parametrizing subschemes $Z\subset A$ such that the sum $\sum_{p\in A}l(\cO_{Z,p})p$ vanishes in the group $A$. Note that $\Kum_m(A)$ is the classical Kummer surface (a $K3$ surface) associated to $A$. In general $\Kum_m(A)$ is a HK manifold of dimension $2m$, and if $m\ge 2$ we have
\begin{equation}
b_2(\Kum_m(A))=7.
\end{equation}
A HK manifold  is of \emph{Type $\Kum_m$} if it is a deformation of $\Kum_m(A)$.  
As in the previous case  if $X$ is a very general HK manifold  of Type $\Kum_m$ then there is no abelian surface (or $2$-dimensional complex torus) $A$ such that $X$ is  isomorphic  (nor birational) to $\Kum_m(A)$. 
\end{expl}
Let $X$ be a HK manifold of dimension $2m$, with a K\"ahler class $\omega$. The geometry of $X$ is in large part determined by the second cohomology group $H^2(X)$ equipped with its Hodge structure and the \emph{Beauville-Bogomolov-Fujiki (BBF) bilinear simmetric form} 
\begin{equation}\label{eccobbf}
H^2(X)\times H^2(X) \xrightarrow{b_X} \CC.
\end{equation}
The following properties characterize $b_X$:
\begin{enumerate}
\item
The restriction of $b_X$ to $H^2(X;\ZZ)\times H^2(X;\ZZ)$ defines an integral bilinear simmetric form which is primitive (non divisible by an integer greater than $1$).
\item
The restriction of $b_X$ to $H^2(X;\RR)\times H^2(X;\RR)$ has signature $(3,b_2(X)-3)$ and is positive definite on the $3$-dimensional subspace
\begin{equation}
\RR[\omega]\oplus \left(H^2(X;\RR)\cap (H^{2,0}(X)\oplus H^{0,2}(X))\right).
\end{equation}
\item
There exists a (Fujiki) constant $c_X\in\QQ_{+}$ such that for $\alpha\in H^2(X;\QQ)$ we have 
\begin{equation}\label{formulafujiki}
\int_X \alpha^{2m}=c_X \cdot(2n-1)!!\cdot q_X(\alpha,\alpha)^m,
\end{equation}
where $q_X(\alpha)\coloneqq b_X(\alpha,\alpha)$ is the value of the quadratic form $q_X$ associated to $b_X$. 
\end{enumerate}
\begin{expl}\label{expl:bbfhilb}
If $X$ is a $K3$ surface then $b_X$ is the intersection form, and it is well-known that $(H^2(X;\ZZ),b_X)\cong U^{\oplus 3}\oplus (-E_8)^{\oplus 2}$, where $U$ is the hyperbolic lattice, $-E_8$ is the unique even unimodular negative definite rank $8$ lattice, and $\oplus$  is orthogonal direct sum. Let 
$\Lambda_{\KKK}$ be this lattice (the \emph{K3 lattice}). 
If $X$ is a HK manifold of Type $K3^{[n]}$, and $n\ge 2$, then 
\begin{equation}\label{bbfk3n}
(H^2(X;\ZZ),b_X)\cong \Lambda_{\KKK}\oplus (-2(n-1)),
\end{equation}
where $\oplus$  is orthogonal direct sum, and $(-2(n-1))$ is the lattice $\ZZ$ with square of a generator equal to $-2(n-1)$. Moreover  $c_X=1$. A more explicit description of  $H^2(S^{[n]};\ZZ)$,  where $S$ is a $K3$ surface, is as follows. Let $\goth_n\colon S^{[n]}\to S^{(n)}$ be  the Hilbert-to-Chow map, defined by $\goth_n([Z])\coloneqq  \sum_{p\in S}l(\cO_{Z,p})p$. Let
$\bm{\mu}_n\colon H^2(S;\ZZ)\to H^2(S^{[n]};\ZZ)$ be the composition of the symmetrization map $H^2(S;\ZZ)\to H^2(S^{(n)};\ZZ)$ and the pull-back  $\goth_n^{*}\colon  H^2(S^{(n)};\ZZ) \to H^2(S^{[n]};\ZZ)$.
  Then
\begin{equation}\label{bbfk3n}
H^2(S^{[n]};\ZZ)=\Im \bm{\mu}_n \oplus \ZZ \delta_n,
\end{equation}
where $2\delta_n$ is the Poincar\'e dual of the divisor  parametrizing non-reduced subschemes (see~\cite{dlmo:cetraro22}).
\end{expl}
\subsubsection{Main results}
The first main result is the following. 
\begin{thm}[Mukai, Huybrechts-G\"ottsche, O'Grady, Yoshioka]\label{thm:vettmukprim}
Let $S$ be a $K3$ surface, and let $v$ be a primitive 
Mukai vector for $S$. Let $h\in\NS(S)_{\QQ}$ be ample and ${\mathsf a}(v)$-generic  (see~\eqref{adivu}). The following hold:
\begin{enumerate}
\item
$\cM_v(S,h)$ is non empty if and only if $v^2 \ge -2$.
\item
Suppose that  $v^2\ge -2$. Then $\cM_v(S,h)$ is a HK (projective) manifold of Type $K3^{[m]}$ where $2m=v^2+2$. 
\item
Suppose that  $\la v,v\ra\ge 2$. A quasi-universal sheaf on $S\times\cM_v(S,h)$ (see~\cite[App.2]{mukai:vbtata}) induces via slant product (this is \emph{Mukai's map}) an isometry of  lattices
\begin{equation}\label{mappamukai}
 v^{\bot}\cap H(S;\ZZ)\xrightarrow{\theta_v} H^2(\cM_v(S,h);\ZZ),
\end{equation}
where the bilinear symmetric form on the left hand side is the restriction of the Mukai pairing, and on the right hand side it is  the BBF bilnear form $b_X$. 

\end{enumerate}
\end{thm}
\begin{cmm}\label{cmm:verfica}
We go over the few statements which follow at once from results that we have already stated, and we give references for the other statements.  Suppose that  $\cM_v(S,h)$ is non empty.
Since  $v$ is primitive, and $h$ is ${\mathsf a}(v)$-generic, it follows from  Corollary~\ref{crl:stabnonstabk3} that $v^2\ge -2$  and  
that $\cM_v(S,h)$ is smooth. Non emptiness if $v^2= -2$ has been proved in~\cite{kuleshov:fvecc}. Let $\cF$ be an $h$ stable torsion-free sheaf with $v(\cF)^2=-2$ (known as \emph{spherical} because the cohomology of $End\cF$ is that of a two-dimensional sphere). It follows easily that $\cF$ is locally-free (if $\cF$ is not locally-free then there is a finite subset  of singular points of $\cF$ and 
$\cF^{\vee\vee}$ is an $h$ slope stable vector bundle, hence there exist non-trivial deformations $\cE$ of $\cF$ with 
$\cE^{\vee\vee}\cong \cF^{\vee\vee}$ and singular points of $\cE$ different from those of $\cF$, this is a contradiction because $\dim\Ext^1_S(\cF,\cF )=0$). If $\cE$ is any  $h$ stable vector bundle on $S$ such that $v(\cE)=v(\cF)$ then $\chi(S,\cE^{\vee}\otimes\cF)=\chi(S,\cE^{\vee}\otimes\cE)=2$. By Serre duality it follows that either there exists a non zero map  $f\colon\cE\to\cF$ or a non zero map  $g\colon\cF\to\cE$. Since  $v(\cE)=v(\cF)$, it follows that $f$ or $g$ is an isomorphsm by $h$ stability of $\cE$, $\cF$. This argument of Mukai, see~\cite[Cor.~3.5]{mukai:vbtata}, has been generalized 
to prove irreducibility of $\cM_v(S,h)$, provided it is non empty,  see~\cite{kaledin-lehn-sorger:spazimodsing}. 
In~\cite{og:accaduespmod} we proved that (2), (3), (4) hold if $v=(r,l,s)$ with $l$ indivisible (from that result one gets easily that  (2), (3), (4) hold if $v=(r,l,s)$ where $r$ and the divisibility of $l$ are coprime). In Subsubsections~\ref{subsubsec:disegnodimuno}, \ref{subsubsec:disegnodimdue} we go over some of the ideas that lead to a proof  if $v=(r,l,s)$ where $r$ and the divisibility of $l$ are coprime.  For a proof valid for all (primitive) $v$ 
see~\cite[Thm.~8.1]{yoshioka:modfvsupab}.
\end{cmm}
\begin{rmk}
If $\la v,v\ra=0$ then a quasi-universal sheaf on $S\times\cM_v(S,h)$  induces via slant product an isometry of  lattices
$\theta_v\colon v^{\bot}\cap H(S;\ZZ)/\ZZ v\xrightarrow{\sim} H^2(\cM_v(S,h);\ZZ)$, see~\cite[Thm.~1.5]{mukai:vbtata}.
\end{rmk}
The second main result deals with the case in which $v$ is not primitive.
\begin{thm}[O'Grady, Kaledin-Lehn-Sorger, Rapagnetta, Perego-Rapagnetta]\label{thm:vettmuknonprim}
Let $S$ be a $K3$ surface. Let $w$ be a primitive 
Mukai vector for $S$ with $\la w,w\ra\ge 2$, and let $v=kw$ where $k$ is an integer and $k\ge 2$. Let $h\in\NS(S)_{\QQ}$ be ample and ${\mathsf a}(v)$-generic  (see~\eqref{adivu}). The following hold:
\begin{enumerate}
\item
$\cM_v(S,h)$ is non empty irreducible, and singular.
\item
$\cM_v(S,h)^{\st}$ is (open) dense in $\cM_v(S,h)$, and the complement is the singular locus of $\cM_v(S,h)$.
\item
 If $w^2=2$ and $k=2$ there exists a crepant desingularization $\wt{\cM}_v(S,h)\to \cM_v(S,h)$, which is a HK (projective) manifold with 
 $b_2(\wt{\cM}_v(S,h))=24$ (whose deformation type is \emph{Type OG10}).
\item
In all other cases, i.e.~if $w^2>2$ or $k>2$,  there exists no crepant desingularization of $\cM_v(S,h)$.
\item
If $w^2>2$ or $k>2$, then $\cM_v(S,h)$ is a primitive symplectic variety whose locally-trivial-deformation type depends only on 
$w^2$ and $k$.
\end{enumerate}
\end{thm}
\subsubsection{Ideas that go into the proof of Theorem~\ref{thm:vettmukprim}: elliptic surfaces}\label{subsubsec:disegnodimuno}
Here $S$ is a (projective) $K3$ surface with an elliptic fibration $\pi\colon S\to\PP^1$ (see Example~\ref{expl:rigidisuk3ell}).
 If $t\in \PP^1$ we let 
$C_t\coloneqq \pi^{-1}(t)$. 
For simplicity we assume that 
$C_t$ is irreducible (and hence reduced) for every $t\in\PP^1$. 
Suppose that $\cF$ is a vector bundle on $S$. If $t\in \PP^1$ we let 
$\cF_t\coloneqq \cF_{|C_t}$. If $C_t$ is smooth we know what it means that $\cF_t$ is (semi)stable or unstable. The analogous definition makes sense also if $C_t$ is singular (we assume that $C_t$ is irreducible to simplify this kind of considerations). 
Let $h\in\NS(S)$   be ample and ${\mathsf a}$-suitable for $\cl(C)$, where $C$ is a fiber of $\pi$.  Let $\cF$ be a torsion-free sheaf on $S$ such that the following hold:
\begin{equation}\label{richieste}
{\mathsf a}(\cF)\le {\mathsf a},\quad\text{$r(\cF)$ and $c_1(\cF)\cdot \cl(C)$ are coprime,}\quad\text{$\cF$ is $h$ semistable.}
\end{equation}
By Proposition~\ref{prp:fvsupfib} the restriction $\cF_t$ to a general fiber $C_t$ is stable. Let $D\subset\PP^1$ be the finite set of $t\in\PP^1$ such that $\cF_t$  is not stable (and hence unstable by~\eqref{richieste}).
Then there is a procedure (semistable reduction) that modifies $\cF$ only along  $\pi^{-1}(D)$, and produces an $h$ slope stable vector bundle $\cE$ on $S$ whose restriction to every elliptic fiber is stable. More precisely $\cE$ fits into an exact sequence
\begin{equation}\label{daeffeae}
0\lra \cE\lra \cF\lra i_{*}\cB\lra 0,
\end{equation}
where $i\colon\bD\hra S$ is the inclusion of a closed subscheme supported on $\pi^{-1}(D)$, and $\cB$ is a sheaf on $\bD$. 
Regarding $\cE$ we have the following result.
\begin{prp}
Let $S$ be a (projective) $K3$ surface with an elliptic fibration $\pi\colon S\to\PP^1$. Let  $\cF$ be a vector bundle on $S$ such that its restriction to a general elliptic fiber is stable. Then  $v(\cF)^2=-2$ if and only if  
$\cF_t$ is stable for all $t\in\PP^1$.  
\end{prp}
\begin{proof}
Suppose that $\cF_t$ is stable for all $t\in\PP^1$. Let $End^0(\cF)\subset End(\cF)$ be the subsheaf of traceless endomorphisms. If  $t\in\PP^1$ then $\cF_t$ is simple because it is 
stable. It follows that $H^p(C_t,End^0\cF_t)=0$ for all $p$, and hence $R^p\pi_{*}End^0\cF$ for all $p$. By the (Grothendieck) spectral sequence of higher direct images abutting to $H^m(S,End^0\cF)$ we get that $H^m(S,End^0\cF)=0$ for all $m$. Thus 
$\chi(S,End\cF)=\chi(S,End^0\cF)+\chi(S,\cO_S)=2$, and hence the equality $v(\cF)^2=-2$ follows from~\eqref{caratmukai}. Now suppose that $v(\cF)^2=-2$. If there exist $t\in\PP^1$  such that $\cF_t$  is not stable then by semistable reduction we produce the vector bundle $\cE$ in~\eqref{daeffeae}. One checks  (see~\cite[(6.2.5)]{og:fascimod}) that $v(\cE)^2< v(\cF)^2=-2$. Let $h\in\NS(S)_{\QQ}$ be ample and 
 ${\mathsf a}(\cF)$-suitable for $\cl(C)$. 
By Proposition~\ref{prp:fvsupfib} 
 $\cE$ is $h$ slope stable and hence $-2\le v(\cE)^2$. This is   a contradiction.
\end{proof}
Suppose that $\cE$ is the rigid vector bundle obtained from $\cF$ as in~\eqref{daeffeae}. Then for a suitable non negative integer $d$ we have an exact sequence
\begin{equation}\label{daeaeffe}
0\lra \cF\otimes\cO_S(-dC)\lra \cE\lra i_{*}\cA\lra 0. 
\end{equation}
In other words $\cF$ is obtained from the unique (up to isomorphism) $h$ stable vector bundle $\cE$ with $v(\cE)=v(\cF)-v(i_{*}\cB)$ 
(unicity is proved in Comment~\ref{cmm:verfica}) via an elementary modification. Note that $v(i_{*}\cB)$ depends only on $v(\cF)$. 
Now fix a Mukai vector $v=(r,l,s)$ such that ${\mathsf a}(v)\le {\mathsf a}$ (see~\eqref{adivu}), and $\gcd(r,l\cdot C)=1$. Then every $h$ semistable sheaf $\cF$ on $S$ is obtained (up to tensorization by a power of the line bundle $\cO_S(C)$) from a single rigid vector bundle $\cE$ via an elementary modification. This quickly leads to a proof that $\cM_v(S,h)$ is non empty and birational to $T^{[m]}$ where $T$ is a $K3$ surface and $2m=v^2+2$, see~\cite{og:accaduespmod,yoshioka:modfvsupab}, provided we know that given a Mukai vector $v$ with $v^2=-2$ there exists a  stable vector bundle $\cE$ on with $v(\cE)^2=-2$ (Kuleshov's Theorem~\cite{kuleshov:fvecc}). Since by Corollary~\ref{crl:stabnonstabk3}  $\cM_v(S,h)$ has a holomorphic symplectic form, it follows that it is a HK manifold. Since  birational HK  manifolds are deformation equivalent by a fundamental result  of Huybrechts, see~\cite{huy:conokahler}, we get that $\cM_v(S,h)$  is a HK manifold of Type $K3^{[m]}$. The argument sketched above gives the following result.
\begin{thm}[=Theorem~(1.0.4) in~\cite{og:accaduespmod}]\label{thm:modfvellsez}
Let $S$ be a (projective) $K3$ surface such that $\NS(S)=\ZZ[C]\oplus\ZZ[\Sigma]$, where $C$ is a fiber of an elliptic fibration 
$\pi\colon S\to\PP^1$ and $\Sigma$ is a section of $\pi$. Let $v=(r,\Sigma+dC,s)$ be a Mukai vector with $v^2\ge -2$, and let 
 $h\in\NS(S)_{\QQ}$   be ample and ${\mathsf a}(v)$-suitable for $\cl(C)$. Then $\cM_v(S,h)$ is a HK (projective) manifold of Type $K3^{[m]}$ where $2m=v^2+2$.
If $v^2\ge 2$ the map in~\eqref{mappamukai} is an isometry of lattices.
\end{thm}
\begin{rmk}
Let $v_r=(r,\Sigma+(r-r^2)C,1-r)$. Then $\cM_{v_r}(S,h)=\{[\cF^r]\}$ where $\cF^r$ is the spherical vector bundle of Example~\ref{expl:rigidisuk3ell}. If $v$ is as in Theorem~\ref{thm:modfvellsez} and $v^2=-2$ then 
$\cM_{v}(S,h)=\{[\cF^r\otimes\cO_S(s+r-1)C]\}$. 
\end{rmk}
\subsubsection{Ideas that go into the proof of Theorem~\ref{thm:vettmukprim}: deformations and polarization}\label{subsubsec:disegnodimdue}
\begin{prp}\label{prp:spalmo}
Let $(S,h)$ and $(S',h')$ be polarized $K3$ surfaces such that $h\cdot h=h'\cdot h'$.  Let $v=(r,h,s)$ and $v'=(r,h',s)$ be Mukai vectors with $v^2=(v')^2\ge -2$, and assume that both $h$ and $h'$ are ${\mathsf a}(v)={\mathsf a}(v')$-generic. If 
Items~(2) and~(3) of Theorem~\ref{thm:vettmukprim} hold for $(S,h)$ and $v$ then they hold  for $(S',h')$ and $v'$.
\end{prp}
\begin{proof}
Recall that we assume that $h,h'$ are primitive. By irreducibility of the moduli space of polarized $K3$ surfaces of degree 
$h\cdot h=h'\cdot h'$ there exist a projective map $f\colon \cX\to T$ with $T$ an irreducible variety,  a relatively ample line bundle $\cL$ on $\cX$, and points $0,0'\in T$ such that $(X_0, c_1(L_0))\cong (S,h)$ and $(X_{0'}, c_1(L_{0'}))\cong (S',h')$ (if $t\in T$ we let $X_t\coloneqq f^{-1}(t)$ and
$L_t\coloneqq\cL_{X_t}$). Note that $(X_t, L_t)$ is a polarized $K3$ surface for every $t\in T$. For $t\in T$ let $v_t$ be the Mukai vector for $X_t$  
given by $v_t\coloneqq(r,c_1(L_t),s)$. Deleting from $T$ a finite subset (which does not contain $0$ nor $0'$) we may assume that 
$c_1(L_t)$ is $v_t$-generic for all $t\in T$. 
By~\cite{maruyama:modfasci} there exists a relative moduli space $g\colon{\mathsf M}\to T$, i.e.~a projective map with $g^{-1}(t)\cong\cM_{v_t}(X_t,c_1(L_t))$  for all $t\in T$. By Corollary~\ref{crl:defcoppia} (this is the key point) the map $g$ is smooth. Since any deformation of HK manifold is a HK manifold (see Remark~\ref{rmk:defhk}), and Mukai's map in~\eqref{mappamukai} is locally constant  this finishes the proof. 
\end{proof}
Let $S$ be an elliptic K3 surface as in Theorem~\ref{thm:modfvellsez}. Then $\cO_S(\Sigma+gC)$ is ample of degree $2g-2$ for all $g\ge 3$. 
Let $h\coloneqq \cl(\Sigma+gC)$, and $v\coloneqq (r,h,s)$ with $v^2\ge -2$. Let $(S',h')$ be a polarized K3 surface of degree $2g-2$, 
 and let $v'$ be the Mukai vector for $S'$ given by $v'\coloneqq (r,h',s)$. Suppose that $h'$ is $v'$-generic. 
By Proposition~\ref{prp:spalmo} and Theorem~\ref{thm:modfvellsez} Items~(2) and~(3) of Theorem~\ref{thm:vettmukprim} hold for $(S',h')$ and $v'$, provided $h$ is $v$-generic and ${\mathsf a}(v)$-suitable. If the two last  conditions are not satisfied  one gets around this obstacle  by tensoring with a suitable line bundle.
The point is that if $\cF$ is an $h$ slope stable torsion-free sheaf on $S$ with $v(\cF)=v$ and $L$ is a line bundle on $S$, then $\cF\otimes L$ 
is  $h$ slope stable and $v(\cF\otimes L)=v(\cF)\cdot e^{c_1(L)}$ whose component in $\NS(S)$ can be made to lie in an arbitrary open 
${\mathsf a}(v)$-chamber. This proves that  Items~(2) and~(3) of Theorem~\ref{thm:vettmukprim} hold whenever the component of $v$ in $\NS(S)$ is equal to $h$. Playing around with tensorization by line bundles one proves the validity of Items~(2) and~(3) in general. For details 
see~\cite[Sect.~2]{og:accaduespmod}

\section{Sheaves  on higher-dimensional hyperk\"ahler varieties}\label{sec:alfinlaciccia}
\subsection{Deformations of sheaves on higher dimensional HK manifolds}\label{subsec:defasci}
\setcounter{equation}{0}
Let $X$ be a higher dimensional HK manifold, i.e.~of dimension greater than $2$. Let $\cF$ be a simple torsion-free sheaf on $X$. Taking our cue from the results  of Subsection~\ref{subsec:deformofasci} we ask the questions:
\begin{enumerate}
\item[(Q1)]
Is  the trace map $\Tr^2\Ext^2_X(\cF,\cF) \to H^2(X,\cO_X)$ injective? 
(See~\eqref{entrambiuno} and the proof of Theorem~\ref{thm:mukart}.)
\item[(Q2)]
Is the deformation space $\Def(\cF)$ smooth? In other words, does the analogue of the first part of 
Theorem~\ref{thm:mukart} hold?
\item[(Q3)]
Is the map $\Def(X,\cF) \to \Def(X,\det\cF)$ smooth, or at least surjective? 
\item[(Q4)]
Suppose in addition that $\cS$ is locally-free. Is the  
map $\Def(\PP(\cF)) \to \Def(X)$  smooth, or at least surjective? (Recall  Corollary~\ref{crl:martysupreme}.)
\end{enumerate}
The answer to (Q1) is negative. An example is provided by the tangent bundle $\Theta_X$. By Yau's Theorem $\Theta_X$ is slope stable with respect to any K\"ahler class. It follows that $\Theta_X$ is simple. A choice of holomorphic symplectic form defines an isomorphism $\Theta_X\xrightarrow{\sim}\Omega_X$. Hence we have
\begin{equation*}
\Ext^2_X(\Theta_X,\Theta_X)=H^2(X,\Theta_X^{\vee}\otimes \Theta_X)\cong H^2(X,\Omega_X\otimes \Omega_X)=
 H^2(X,\Sym^2\Omega_X)\oplus H^2(X,\Omega^2_X).
\end{equation*}
Thus $\dim \Ext^2_X(\Theta_X,\Theta_X)\ge h^{2,2}(X)$. Wedge product defines an injection $\Sym^2 H^{1,1}(X)\hra H^{2,2}(X)$, and hence 
\begin{equation}\label{disegextdue}
\dim \Ext^2_X(\Theta_X,\Theta_X)\ge {b_2(X)-1\choose 2}.
\end{equation}
Unless $X$ is a hypothetical HK manifold with $b_2(X)=3$, the right-hand side of~\eqref{disegextdue} is greater than $1$, and it follows that $\Tr^2$ is not injective.

The answer to (Q2) is negative as well, if no additional condition is imposed on $\cF$ . For any $l\ge 4$ there exist analytic (non reduced) subsets $Z\subset X$ of finite length $l$ such that the Douady space $X^{[l]}$ is singular at $[Z]$. Since the germ of $X^{[l]}$  at $[Z]$ is isomorphic to the deformation space $\Def(\cI_Z)$ of the ideal sheaf $\cI_Z$, the latter is not smooth. Examples with $\cF$ locally free simple (even slope stable) appear in  Subsection~\ref{subsec:lavoroinfinito}. A conjecture of Beckmann (see Conjecture~\ref{cnj:atomicoliscio}) states that the answer to (Q2) is positive if $\cF$ is a slope stable atomic  (see Definition~\ref{dfn:eccoatomico}) vector bundle.

Regarding questions (Q3) and (Q4): if $\cF$ is a projectively hyperholomorphic vector bundle, then the picture resembles that in dimension $2$,  see Subsection~\ref{subsec:fibrativerb}.
\subsection{Modular  sheaves}
\setcounter{equation}{0}
\subsubsection{Definition of modular sheaves, and first examples}
In the present subsection $X$ is a HK manifold. Let $\cF$ be a (non zero) torsion-free sheaf on $X$. If $h$ is an ample class on $X$ then $h$ slope (semi)stability of $\cF$ may be expressed in terms of the BBF bilinear form $b_X$ on the   N\'eron-Severi group $\NS(X)$. Thus one may hope that variation of slope (semi)stability may behave as in the case of surfaces. This is  the case if one puts a topological restriction on the discriminant of  $\cF$.  
\begin{dfn}\label{dfn:effemod}
Let $X$ be a HK manifold of dimension $2n$. A torsion-free sheaf $\cF$ on $X$ is \emph{modular} if 
there exists $d(\cF)\in\QQ$ such that   for all $\alpha\in H^2(X)$ we have
\begin{equation}\label{fernand}
\int_X \Delta(\cF)\cdot \alpha^{2n-2}=d(\cF) \cdot (2n-3)!! \cdot q_X(\alpha)^{n-1},
\end{equation}
 where $q_X(\alpha)=b_X(\alpha,\alpha)$ is the square for the BBF bilinear symmetric form. 
\end{dfn}
Note that any torsion-free sheaf on a K3 surface is modular. 
\begin{rmk}\label{semphodge}
Let $X$ be a HK manifold of dimension $2n$. Let $D(X)\subset H(X)$  be the image of the map $\Sym H^2(X)\to H(X)$ defined by cup-product. Let $D^i(X):=D(X)\cap H^i(X)$. The  pairing $D^i(X)\times D^{4n-i}(X)\to\CC$ defined by intersection product is non degenerate~\cite{verbcohk,bogcohk,rhag}, hence there is a splitting $H(X)= D(X)\oplus D(X)^{\bot}$, where orthogonality is with respect to the intersection pairing. Let $p_D\colon H(X)\to D(X)$ be the corresponding orthogonal projection. 
A torsion-free sheaf  $\cF$ on $X$ is modular if and only if  $p_D(\Delta(\cF))$  is a multiple of $p_D(c_2(X))$.
In particular if  $\Delta(\cF)$ is a multiple of $c_2(X)$ then $\cF$ is modular. 
\end{rmk}
\begin{expl}\label{expl:fvmodsubd}
Let $Y\subset\PP^5$ be a  a smooth cubic hypersurface, and let $X\subset\GR(1,\PP^5)$ be the variety of lines contained in $Y$.
 Let $h\in \NS(X)$ be the Pl\"ucker polarization. Then $X$ is a (projective) HK manifold of Type $K3^{[2]}$, see~\cite{beaudon}. Let $\cQ$  be  the restriction to $X$ of the tautological rank $4$  quotient vector bundle on $\GR(1,\PP^5)$. Then (see~\cite[Subsect.2.1]{og:fascimod})
\begin{equation}\label{yaris}
\ch_0(\cQ)  =  4, \quad \ch_1(\cQ)  =  h, \quad \ch_2(\cQ)  =  \frac{1}{8}\left(h^2 -c_2(X)\right).
\end{equation}
 Thus $\Delta(\cQ)=c_2(X)$, and hence $\cQ$ is modular.  
 Let $\cU$ be the restriction to $X$ of the tautological rank $2$ subbundle on $\GR(1,\PP^5)$.  Then $\Delta(\cU)=(3h^2-c_2(X))/2$, and hence $\cU$ is \emph{not} modular.
\end{expl}
\begin{expl}\label{expl:fvmodsudv}
Let $X\subset \GR(5,\PP^9)$ be a smooth Debarre-Voisin variety, see~\cite{devo}.  Let $h\in \NS(X)$ be the Pl\"ucker polarization. 
 Then $X$ is a  (projective) HK manifold of Type $K3^{[2]}$. Let 
$\cQ$ be  the restriction to $X$ of the tautological rank $4$  quotient vector bundle on $\GR(5,\PP^9)$. Then   the equalities in~\eqref{yaris} hold. 
In fact they follow from~\cite[Proof of Lemma 4.5, 2nd eqtn after (12), p.83]{devo} 
 Thus  $\cQ$ is modular. Let $\cU$ be the restriction to $X$ of the tautological rank $6$ subbundle on $\GR(5,\PP^9)$. 
 Then $\Delta(\cU)=(5h^2-3c_2(X))/2$, and hence $\cU$ is \emph{not} modular.
\end{expl}
\begin{expl}\label{expl:fatighenti}
Let $X$ be a HK manifold, and let $\cF$ be a modular vector bundle on $X$, of rank $r$. Let $\alpha=(\alpha_1,\ldots,\alpha_r)$ be a partition, and let $S^{\alpha}\cF$ be the vector bundle on $X$ obtained by applying the Schur functor $S^{\alpha}$ to $\cF$. For example, by suitable choices of $\alpha$  we obtain the symmetric and exterior products $\Sym^a\cF$ and $\bigwedge^a\cF$. Then $S^{\alpha}\cF$ is modular,
 see~\cite{fatighenti-et-alii:discreschur}.
\end{expl}
\subsubsection{Bridgeland-King-Reid, and \lq\lq synthetic\rq\rq\ modular sheaves}
Let $S$ be a smooth projective surface. The symmetric group $\Sigma_n$ on $\{1,\ldots,n\}$ acts  on $S^n$ by permutation of the factors. Let $\Coh_{\Sigma_n}(S^n)$ be the category of $\Sigma_n$-equivariant coherent sheaves on $S^n$, and let 
$D^b_{\Sigma_n}(S^n)$ be  the bounded derived category of $\Coh_{\Sigma_n}(S^n)$.
The Bridgeland-King-Reid (BKR) correspondence for the crepant resolution $\goth_n\colon S^{[n]}\to S^{(n)}$ gives an equivalence between 
$D^b(S^{[n]})$ and $D^b_{\Sigma_n}(S^n)$, and it 
 gives a method for producing sheaves on $S^{[n]}$ whose cohomology can be computed. We follow  Krug's version of the BKR correspondence, see~\cite{krug-bkr}. 
Let $X_n(S)$ be the isospectral Hilbert scheme of $n$ points on $S$,  introduced and studied by Haiman,  
 see Definition 3.2.4 in~\cite{haiman:polygraphs}. We have a commutative diagram
\begin{equation}\label{commiso}
\xymatrix{ X_n(S)\ar[d]_{\rho}\ar[rr]^{\tau}    &  &  S^n \ar[d]^{\pi}\\ 
  S^{[n]}  \ar[rr]^{\gamma} & & S^{(n)} }
\end{equation}
and $X_n(S)$ is the reduced scheme associated to the fiber product of $S^n$ and $S^{[n]}$ over $S^{(n)}$. 
Moreover (see Corollary 3.8.3 in~\cite{haiman:polygraphs}) the map $\tau$ is identified with the blow up of $S^n$ with center the big diagonal. The symmetric group $\Sigma_n$ acts on $X_n(S)$ and on $S^n$, and $\tau$ is equivariant for these actions. 
Let $\cA\in \Coh_{\Sigma_n}(S^n)$. Then $\Sigma_n$  acts on the pull-back $\tau^{*}(\cA)$. Since  the action of $\Sigma_n$ on  
$X_n(S)$ commutes with $\rho$, $\Sigma_n$  also acts on the push-forward 
$\rho_{*}\tau^{*}(\cA)$, and hence the sheaf of $\Sigma_n$-invariants $\left(\rho_{*}\tau^{*}(\cA)\right)^{\Sigma_n}$ makes sense.
\begin{hyp}\label{hyp:riscorta}
$\cA$  has a short locally-free resolution $0\lra \cF_1\lra\cF_0\lra\cA\lra 0$.
\end{hyp}
If Hypothesis~\ref{hyp:riscorta} holds  then $\tau^{*}(\cA)$ equals the derived pull-back of $\cA$, and hence  
$\left(\rho_{*}\tau^{*}(\cA)\right)^{\Sigma_n}$ is the  object in $D^b(S^{[n]})$ which correspond to $\cA$ under the BKR equivalence, 
i.e.~$\BKR(\cA)$. In particular   Ext-groups of sheaves on $S^{[n]}$ are described in terms of Ext-groups of sheaves on $S^{n}$. Note that if   $\cF,\cG\in \Coh_{\Sigma_n}(S^n)$ then $\Sigma_n$ acts on the Ext-group 
$\Ext^p(\cF,\cG)$ (any $p$), and hence the group of invariants $\Ext^p(\cF,\cG)^{\Sigma_n}$ makes sense. 
The BKR equivalence
gives the following result.
\begin{prp}[BKR]\label{prp:extbkr}
Suppose that  $\cF,\cG\in \Coh_{\Sigma_n}(S^n)$, and that Hypothesis~\ref{hyp:riscorta} holds for $\cA=\cF$ and  $\cA=\cG$. Then the BKR correspondence defines an isomorphism
\begin{equation}
\Ext^p(\cF,\cG)^{\Sigma_n}\xrightarrow{\sim}\Ext^p(\left(\rho_{*}\tau^{*}(\cA)\right)^{\Sigma_n},\left(\rho_{*}\tau^{*}(\cB)\right)^{\Sigma_n}).
\end{equation}
\end{prp}
We give two examples of the above construction.   
\begin{expl}\label{expl:fibvettrigidi}
Let $S$ be a $K3$ surface.  Let $\cE$ be a  locally-free sheaf on $S$. The locally-free sheaf on $S^n$ given by
$\cE\boxtimes\ldots\boxtimes\cE$  has an action by $\Sigma_n$, given by pull-back. We get a corresponding 
$\Sigma_n$-action on $\tau^{*}(\cE\boxtimes\ldots\boxtimes\cE)$, and a $\Sigma_n$-action on 
$\rho_{*}\tau^{*}(\cE\boxtimes\ldots\boxtimes\cE)$. Let 
\begin{equation}
\cE[n]^{+}\coloneqq \rho_{*}\tau^{*}(\cE\boxtimes\ldots\boxtimes\cE)^{\Sigma_n}
\end{equation}
 be the sheaf on $S^{[n]}$ given by the invariants for the  $\Sigma_n$-action. 
The map $\rho$ in~\eqref{commiso} is finite,  and it is flat because 
$X_n(S)$ is CM by~\cite[Thm.~3.1]{haiman:polygraphs}. It follows that $\rho_{*}\tau^{*}(\cE\boxtimes\ldots\boxtimes\cE)$
 is  locally-free, and hence also $\cE[n]^{+}$ is  locally-free. Note that 
\begin{equation}
r(\cE[n]^{+})=r(\cE)^n. 
\end{equation}
Applying  Proposition~\ref{prp:extbkr} one gets isomorphisms 
\begin{eqnarray}
\scriptstyle
\Sym^n H^0(S,\cE^{\vee}\otimes\cE) & \scriptstyle \xrightarrow{\sim} & \scriptstyle\Ext^0_{S^{[n]}}( \cE[n]^{+},\cE[n]^{+}), \label{extpiu:uno}\\
\scriptstyle\Sym^{n-1} H^0(S,\cE^{\vee}\otimes\cE)\otimes H^1(S,\cE^{\vee}\otimes\cE) & \scriptstyle\xrightarrow{\sim} & \scriptstyle \Ext^1_{S^{[n]}}( \cE[n]^{+},\cE[n]^{+}),\label{extpiu:due}\\ 
\scriptstyle\Sym^{n-1} H^0(S,\cE^{\vee}\otimes\cE)\otimes H^2(S,\cE^{\vee}\otimes\cE)\oplus 
\Sym^{n-2} H^0(S,\cE^{\vee}\otimes\cE)\otimes \bigwedge^2 H^1(S,\cE^{\vee}\otimes\cE) & \scriptstyle\xrightarrow{\sim} & \scriptstyle \Ext^2_{S^{[n]}}( \cE[n]^{+},\cE[n]^{+}).\label{extpiu:tre}
\end{eqnarray}
Suppose in addition that $\cE$ is spherical, i.e.~that
\begin{equation}
h^0(S,\cE^{\vee}\otimes\cE)=h^2(S,\cE^{\vee}\otimes\cE)=1,\qquad h^1(S,\cE^{\vee}\otimes\cE)=0.
\end{equation}
Then $\cE[n]^{+}$ is modular because (see~\cite[Prop.~3.2]{og:fascik3n})
\begin{equation}\label{cidue}
\Delta(\cE[n]^{+})   =  \frac{r(\cE)^{2n-2}(r(\cE)^2-1)}{12}c_2(S^{[n]}). 
\end{equation}
In fact the above formula should hold under the weaker hypothesis that $\chi(S,\cE^{\vee}\otimes\cE)=2$ (if $n=2$ this is proved, 
see~\cite[Prop.~5.4]{og:fascimod}). Moreover we believe that the following converse holds: if $\cE[n]^{+}$ is modular then $\chi(S,\cE^{\vee}\otimes\cE)=2$. For $n=2$ it follows from the computations in~\cite[Eqtn.~(5.4.7)]{og:fascimod}. 
\end{expl}
\begin{prp}
Let  $\cE$ be a spherical vector bundle on a $K3$ surface $S$. Then $\cE[n]^{+}$ is a simple, rigid modular vector bundle on $S^{[n]}$, and the forgetful maps
\begin{equation}\label{scordarello}
\Def(S^{[n]},\cE[n]^{+}) \lra \Def(S^{[n]},\det\cE[n]^{+}),\qquad \Def(S^{[n]},\PP(\cE[n]^{+})) \lra \Def(S^{[n]})
\end{equation}
are smooth.
\end{prp}
\begin{proof}
We have already mentioned that $\cE[n]^{+}$ is a  modular vector bundle. 
It is simple by~\eqref{extpiu:uno}, and rigid by~\eqref{extpiu:due}. We have  $\dim \Ext^2_{S^{[n]}}( \cE[n]^{+},\cE[n]^{+})=1$ by~\eqref{extpiu:tre}. The first map in~\eqref{scordarello} is smooth by Theorem~\ref{thm:defcoppia}, the second one is smooth by  Theorem~\ref{thm:evvivahorikawa}.
\end{proof}

\begin{expl}\label{expl:moltimoduli}
Let $S$ be a $K3$ surface. Let $\cE_1,\ldots,\cE_n$ be  sheaves on $S$, locally free, with the possible exception of one which is torsion-free. 
  The torsion-free sheaf on $S^n$ given by
\begin{equation*}
\bigoplus\limits_{\sigma\in\Sigma_n}\cE_{\sigma(1)}\boxtimes\ldots\boxtimes
\cE_{\sigma(i)}\boxtimes\ldots\boxtimes\cE_{\sigma(n)}
\end{equation*}
  has an action by $\Sigma_n$, given by pull-back. Let $\cG(\cE_1,\ldots,\cE_n)$ be the sheaf on $S^{[n]}$ corresponding to the above sheaf via the BKR equivalence. Thus
\begin{equation}\label{gieunoedue}
\cG(\cE_1,\ldots,\cE_n)=\left(\rho_{*}\tau^{*}\bigoplus\limits_{\sigma\in\Sigma_n}\cE_{\sigma(1)}\boxtimes\ldots\boxtimes
\cE_{\sigma(i)}\boxtimes\ldots\boxtimes\cE_{\sigma(n)}\right)^{\Sigma_n}. 
\end{equation}
Note that $\cG(\cE_1,\ldots,\cE_n)$  is torsion-free of rank $n!\cdot r(\cE_1)\cdot\ldots\cdot r(\cE_n)$, and that  if all the $\cE_i$'s are locally free then  
it is  locally free (because 
 $\rho$ is finite and flat by~\cite[Thm.~3.1]{haiman:polygraphs}).  If
\begin{equation}\label{giusteipotesi}
r_j c_1(\cE_i)=r_i c_1(\cE_j)\quad \forall i,j\in\{1,\ldots,n\},\ \ \text{and}
\ \  \sum_{i=1}^n  \frac{v(\cE_i)^2}{r_i^2}=0
\end{equation}
then $\cG(\cE_1,\ldots,\cE_n)$ is modular, because (see~\cite[Prop.~2.7,Rmk~2.10]{og:highdim}) 
\begin{equation}\label{discingen}
\Delta(\cG(\cE_1,\ldots,\cE_n))  =  \frac{(n!)^2 r(\cE_1)^2\cdot\ldots\cdot r(\cE_n)^2}{12} c_2(S^{[n]}),
\end{equation}
\end{expl}
\subsubsection{Modular sheaves on Lagrangian HK manifolds}
Suppose that $\pi\colon X\to B$ is a surjective map from a HK manifold of dimension $2n$ to a  projective variety $B$ such that  
$0<\dim B<\dim X$ and $\pi_{*}\cO_X=\cO_B$. For  $t\in B$ let  $X_t\coloneqq \pi^{-1}(t)$.
Then $\dim B=n$, and if $X_t$ is smooth ($t\in B$)  the following hold.
First $X_t$ is a Lagrangian abelian subvariety of $X$ (see~\cite{matsushita:hkfibrate,matsushita:hkfiblagr}), and secondly there exists (see~\cite{wieneck1}) an ample 
class $\theta_t\in \NS(X_t)$ (even if $X$ is non projective!) such that
\begin{equation}\label{classeteta}
\Im\left(H^2(X;\ZZ)\to H^2(X_t;\ZZ)\right)=\ZZ\theta_t.
\end{equation}
Lastly, if $B$ is smooth then it is isomorphic to $\PP^n$, see~\cite{hwang:basefibr}.
\begin{expl}\label{expl:essennelagr}
Let $S$ be a (projective) K3 surface with an elliptic fibration $f\colon S\to\PP^1$ (note: this does not mean that there is a section of   $f\colon S\to\PP^1$). If $n\in\NN_{+}$ then $S^{[n]}$ is a (projective) HK manifold of dimension $2n$, see Example~\ref{expl:esempihk}. Let  
\begin{equation}
\begin{matrix}
S^{[n]} & \xrightarrow{\pi} & (\PP^1)^{(n)}\cong\PP^n \\
[Z] & \mapsto & \sum_{x\in\PP^1}\left(\sum_{y\in f^{-1}(x)}l(\cO_{Z,y})\right)x
\end{matrix}
\end{equation}
If $x_1,\ldots, x_n$ are pairwise distinct, then  $\pi^{-1}(x_1+\ldots+ x_n)$ is  the abelian variety  $C_{x_1}\times\ldots\times C_{x_n}$.
\end{expl}
We also recall the notion of a semi-homogeneous vector bundle. Let  $X$ be   an abelian variety.
Let $\Aut^0(X)$ be the group of translations of $X$.  Thus $\Aut^0(X)$ is a (bona fide) abelian variety isomorphic to $X$ as algebraic variety.
 For $a\in \Aut^0(X)$, let $T_a\colon X\to X$ be the corresponding translation.  
  A vector bundle  $\cF$ on $X$ is  \emph{semi-homogeneous} (see~\cite{muksemi}) if, for every $a\in \Aut^0(X)$, there exists a line bundle $\cL$ on $X$ (which depends on $a$) such that $T_a^{*}\cF\cong \cF\otimes\cL$.   
\begin{thm}[Mukai]\label{thm:vivasemihom}
Let  $X$ be  an abelian variety. 
\begin{enumerate}
\item
If $\cF$ is  a simple semi-homogeneous  vector bundle on $X$ then $H^1(X,End^0(\cF))=0$. 
\item
If $\cF,\cG$ are simple semi-homogeneous  vector bundles on $X$ with $c_1(\cF)/r(\cF)=c_1(\cG)/r(\cG)$ (in $\NS(X)_{\QQ}$), then 
there exists a line bundle $\cL$ on $X$ such that $\cF\cong\cG\otimes\cL$.
\end{enumerate}
\end{thm}
\begin{proof}
Item~(1) is contained in~\cite[Thm.~5.8]{muksemi}, Item~(2) is contained in~\cite[Prop.~6.17]{muksemi}.
\end{proof}
\begin{rmk}\label{rmk:semihomedell}
Let $\dim X=1$, i.e.~$X$ is a curve of genus $1$. Then every stable vector bundle $\cF$ on $X$ is semi-homogeneous because up to isomorphism  $\cF$ is determined by $(r(\cF),\det \cF)$. Now assume that $\dim X\ge 2$. A general stable vector bundle  on $X$ is not 
semi-homogeneous. Stable semi-homogeneous vector bundles behave  as stable vector bundles on a curve of genus $1$, up to a point. A meaningful difference:  stable semi-homogeneous  vector bundles $\cF,\cG$ on $X$ such that 
\begin{equation}\label{rangodet}
r(\cF)=r(\cG),\qquad \det\cF\cong\det \cG
\end{equation}
 are not necessarily isomorphic.  In fact let $g\coloneqq\dim X$ and $r\coloneqq r(\cF)=r(\cG)$. By~\cite[Prop.~7.1]{muksemi} it follows  that, up to isomorphism,   there are  $r^{2g-2}$ distinct stable semi-homogeneous  vector bundles $\cG$ on $X$ such that~\eqref{rangodet} holds.
\end{rmk}
The result below shows that modular sheaves on HK manifolds with a  Lagrangian fibration behave very much like sheaves on K3 surfaces with an elliptic fibration. 
\begin{prp}[Prop.~2.7 in~\cite{og:fascimod}]\label{prp:ressemi}
Let  $\pi\colon X\to\PP^n$ be a   Lagrangian fibration of a HK manifold of dimension $2n$, and let $\cF$ be a modular vector bundle on $X$. Suppose that $t\in\PP^n$ is a regular value of $\pi$ and that $\cF_{|X_t}$ is $\theta_t$ (see~\eqref{classeteta}) slope-stable. Then 
$\cF_{|X_t}$  is  semi homogeneous.
\end{prp}
\subsection{HK slope stability, and its variation  for modular  sheaves}
\setcounter{equation}{0}
\subsubsection{HK slope (semi)stability}
Let $X$ be a HK manifold. Recall that $b_X$ is the BBF symmetric bilinear form on $H^2(X)$, see~\eqref{eccobbf}. Let $h\in\NS(X)_{\QQ}$.   If $\cF$ is a sheaf on $X$ of positive rank the \emph{$h$ HK-slope} of $\cF$ is given by
\begin{equation}
\mu_h^{HK}(\cF)\coloneqq b_X\left(\frac{c_1(\cF)}{r(\cF)}, h\right).
\end{equation}
\begin{dfn}\label{dfn:stabnef}
Let $\cF$  be a torsion-free sheaf on $X$. Then $\cF$ is \emph{$h$ HK-slope semistable} if for all  subsheaves $\cE\subset\cF$ with $0<r(\cE)<r(\cF)$ we have $\mu_h^{HK}(\cE)\le \mu_h^{HK}(\cF)$.
If strict inequality holds for all such $\cE$ then $\cF$ is \emph{$h$  HK-slope stable}. 
\end{dfn}
Let $h\in\NS(X)_{\QQ}$. Polarizing Fujiki's formula~\eqref{formulafujiki} one gets that 
\begin{equation*}
\int_X \frac{c_1(\cF)}{r(\cF)}\cdot h^{2n-1}=c_X  \cdot(2n-1)!! \cdot b_X\left(\frac{c_1(\cF)}{r(\cF)},h\right)\cdot q_X(h)^{n-1},
\end{equation*}
where $\dim X=2n$. Suppose that $h$ is ample. Since $q_X(h)>0$ (see Subsubsection~\ref{subsubsec:accakappa}) it follows that $\cF$ is  $h$ HK-slope (semi)stable   if and only if it is $h$ slope (semi)stable. 
\begin{expl}\label{expl:hkstablag}
Let $X$ be a  HK manifold of dimension $2n$ with a Lagrangian fibration $\pi\colon X\to\PP^n$. For $t\in\PP^n$ let $X_t\coloneqq \pi^{-1}(t)$. Let 
$f\in\NS(X)$ be given by $f\coloneqq\pi^{*}c_1(\cO_{\PP^n}(1))$.  
Let $\cF$  be a torsion-free sheaf  on $X$. Then $\cF$ is $f$ HK-slope (semi)stable   if and only if  for every subsheaf $\cE\subset\cF$ with $0<r(\cE)<r(\cF)$ we have 
$\mu_{\theta_t}(\cE_t)<\mu_{\theta_t}(\cF_t)$ (respectively $\mu_{\theta_t}(\cE_t)\le\mu_{\theta_t}(\cF_t)$)  where $t\in\PP^n$ is general, 
$\cE_t=\cE_{|X_t}$, 
$\cF_t=\cF_{|X_t}$ and $\theta_t$ is the ample class appearing in~\eqref{classeteta}. This follows from Fujiki's formula~\eqref{formulafujiki}, see~\cite[Lemma~3.11]{og:fascimod}, more precisely the equalities in~(3.5.3) loc.~cit. 
\end{expl}
\begin{rmk}\label{rmk:pazzagioia}
Let $\cF$ be a torsion-free sheaf on a HK manifold $X$.
Then $\cF$  is $h$ HK-slope semistable if for all  subsheaves $\cE\subset\cF$ with 
$0<r(\cE)<r(\cF)$ we have 
$b_X(\lambda_{\cE,\cF},h)\le 0$, where $\lambda_{\cE,\cF}$ is as in~\eqref{eccolambda}, and it is $h$ HK-slope stable if strict inequality holds for all such $\cE$. 
\end{rmk}
\subsubsection{Variation of HK slope (semi)stability for modular sheaves}
The results are  analogous to those in Subsection~\ref{subsec:varstab}. Below is an analogue of the number defined by~\eqref{esmeralda}.
\begin{dfn}\label{dfn:adieffe}
Let $X$ be a HK manifold, and let $\cF$ be a modular torsion-free sheaf  on $X$.
Then
\begin{equation}\label{larissa}
{\mathsf a}(\cF):=\frac{r(\cF)^2 \cdot d(\cF) }{4c_X},
\end{equation}
where $d(\cF)$ is as in Definition~\ref{dfn:effemod}, and $c_X$ is the Fujiki constant of $X$ (see~\eqref{formulafujiki}).
\end{dfn}
Below is the analogue of Proposition~\ref{prp:lizzy}.
\begin{prp}[Prop.~3.10 of~\cite{og:fascimod}]\label{prp:lizzydue}
Let $(X,h)$ be a polarized HK manifold. Let $\cF$ be an $h$ slope semistable   torsion-free sheaf  on $X$. Suppose that $\cE\subset\cF$ is a subsheaf such that $0<r(\cE)<r(\cF)$ and $\mu_h(\cE)=\mu_h(\cF)$. Then 
\begin{equation*}
-{\mathsf a}(\cF)\le q_X(\lambda_{\cE,\cF})\le 0,
\end{equation*}
and the right inequality is an equality  only if $\lambda_{\cE,\cF}=0$.
\end{prp}
\begin{rmk}
 The condition of modularity is essential  for the proof of Proposition~\ref{prp:lizzydue}.
\end{rmk}
Below are the analogues of Definition~\ref{dfn:poladatta}, Remark~\ref{rmk:fratelli}, and Proposition~\ref{prp:paragonestab}. Once 
Proposition~\ref{prp:lizzydue} has been proved, they follow by arguments similar to those given for surfaces.
\begin{dfn}\label{dfn:helen}
Let $X$ be a HK manifold, and let  $h_0\in\Nef(X)_{\QQ}$. 
Let  ${\mathsf a}> 0$. An ample class $h\in\Amp(X)_{\QQ}$ is \emph{${\mathsf a}$-suitable for $h_0$}  if  for all
 $\lambda\in \NS(X)$  such that $ -{\mathsf a}\le q_X(\lambda)<0$ one of the following holds:
\begin{enumerate}
\item
$b_X(\lambda,h)>0$ and $b_X(\lambda,h_0)\ge 0$.
\item
 $b_X(\lambda,h)=0$  and $b_X(\lambda,h_0)= 0$.
\item
 $b_X(\lambda,h)<0$  and $b_X(\lambda,h_0)\le 0$.
\end{enumerate}
\end{dfn}
\begin{rmk}
   Let $X$ be a projective HK manifold, and let ${\mathsf a}> 0$. 
Let $h_0,h\in\NS(X)_{\QQ}$ with $h_0$  nef and $h$ ample. Thus $Nh_0+h$ is ample for any positive integer $N$. 
If $N$ is large  then $Nh_0+h$ is ${\mathsf a}$-suitable for $h_0$. If $q_X(h_0)>0$ this follows from the analogue of the wall-and-chamber decomposition of  $\Amp(X)$ described in
Subsubsection~\ref{subsubsec:muricamere}, if $q_X(h_0)=0$
 see~\cite[Prop~5.13]{og:highdim}.
\end{rmk}
\begin{prp}\label{prp:mirren}
Let $X$ be a projective HK manifold, and let  $h_0\in\NS(X)_{\QQ}$ be nef. 
Let  ${\mathsf a}> 0$, and let $h\in\NS(X)_{\QQ}$  be ample and ${\mathsf a}$-suitable for $h_0$. 
 Suppose that  $\cF$ is  a torsion-free  sheaf  on $X$ which  is
  $h_0$ slope semistable  but not $h$ slope stable, and
   that ${\mathsf a}(\cF)\le {\mathsf a}$.  
    Then there exists a subsheaf  $\cE\subset \cF$ such that 
\begin{equation}
0<r(\cE)<r(\cF),\qquad b_X(\lambda_{\cE,\cF}, h)\ge 0,\qquad b_X(\lambda_{\cE,\cF}, h_0)= 0.
\end{equation}
\end{prp}
Proposition~\ref{prp:mirren} gives a wall-and-chamber decomposition of $\Amp(X)_{\RR}$ (recall Remark~\ref{rmk:pazzagioia}) with properties analogous to those 
of the wall-and-chamber decomposition discussed in Subsubsection~\ref{subsubsec:muricamere}, see~\cite[Prop.~3.4]{og:fascimod}. 
It also gives the following analogue of Proposition~\ref{prp:fvsupfib}.
\begin{prp}[Prop.~5.9 of~\cite{og:highdim}]\label{prp:lagstab}
Let $\pi\colon X\to\PP^n$ be a  Lagrangian fibration of a projective HK manifold, and let $f\coloneqq \pi^{*}c_1(\cO_{\PP^n}(1))$. Let 
${\mathsf a}>0$, and let $h\in\NS(X)_{\QQ}$ 
 be ample and ${\mathsf a}$-suitable for $f$. Let  $\cF$ be  a torsion-free  modular sheaf  on $X$ such that ${\mathsf a}(\cF)\le {\mathsf a}$.  
\begin{enumerate}
\item[(a)]
 Suppose that $\cF$ is not $h$-slope stable, and that   $\cF_t$ is $\theta_t$ slope semistable for general $t\in\PP^n$. Then there exists a subsheaf  $\cH\subset \cF$ with $0<r(\cH)<r(\cF)$ such that $\mu_h(\cH)\ge\mu_h(\cF)$ and 
 $\mu_{\theta_t}(\cH_t)=\mu_{\theta_t}(\cF_t)$ for a general regular value $t\in\PP^n$ ($\theta_t$ as in~\eqref{classeteta}).
\item[(b)]
If  $\cF$ is $h$-slope stable, then $\cF_t$ is $\theta_t$ slope semistable 
 for a general  regular value $t\in\PP^n$. 
\end{enumerate}
\end{prp}
\begin{crl}\label{crl:lagstab}
Let $\pi\colon X\to\PP^n$ be a  Lagrangian fibration of a  projective HK manifold. Let ${\mathsf a}>0$, and let
 $h\in\NS(X)_{\QQ}$ 
 be ample and ${\mathsf a}$-suitable. Let  $\cF$ be  a torsion-free  modular sheaf  on $X$ such that 
 ${\mathsf a}(\cF)\le {\mathsf a}$.  
If $\cF_t$ is  $\theta_t$ slope stable for a general regular value  $t\in\PP^n$
then  $\cF$ is $h$ slope stable. 
\end{crl}
\subsection{(Projectively) hyperholomorphic vector bundles}\label{subsec:fibrativerb}
\setcounter{equation}{0}
\subsubsection{Twistor families of HK manifolds}
Let $X$ be a HK manifold of dimension $2n$, with K\"ahler class  $\omega\in H^{1,1}_{\RR}(X)$. One associates to  $(X,\omega)$ a \emph{twistor family} $f\colon \cX(\omega)\to\PP^1$ of complex manifolds (i.e.~$f$ is a smooth submersive proper map of complex manifolds) with the following properties: 
\begin{enumerate}
\item[(a)]
The fiber $X_0\coloneqq f^{-1}(0)$ is isomorphic to $X$.
\item[(b)]
The composition  of Kodaira-Spencer and contraction with a holomorphic symplectic form
\begin{equation}
\Theta_{\PP^1}(0)\xrightarrow{\ks} H^1(X,\Theta_X)\overset{\overset{\sigma}{\sim}}{\lra} H^1(X,\Omega_X),
\end{equation}
 has image generated by $\omega$.
\end{enumerate}
Twistor families are fundamental objects in the theory of HK manifolds (\cite{huy:hkbasic,huy:hkbasicerratum,huy:conokahler,huy:globtor,verb:globtor}). 
\subsubsection{(Projectively) hyperholomorphic vector bundles}\label{subsubsec:iperolom}
Below we work in the category of analytic spaces. In particular vector bundles are holomorphic, and sheaves are analytic and coherent. 
Let $X$ be a compact K\"ahler  manifold of dimension $d$, with K\"ahler class $\omega$.  The $\omega$-\emph{slope of} a (non-zero) torsion-free sheaf 
$\cF$ on $X$
 is given by the formula in~\eqref{pendenza} with $c_1(L)$ replaced by $\omega$:
\begin{equation}
\mu_{\omega}(\cF)\coloneqq
\int_X\frac{c_1(\cF) \wedge \omega^{d-1}}{r(\cF)}.
\end{equation}
Let $\cF$ be a (non-zero) torsion-free sheaf 
 on $X$. Then $\cF$ is $\omega$-\emph{slope semistable} if for all (non-zero) torsion-free subsheaves $\cE\subset\cF$ one has 
$\mu_{\omega}(\cE)\le \mu_{\omega}(\cF)$, and it is $\omega$-\emph{slope stable} if strict inequality holds for all torsion-free subsheaves 
$\cE\subset\cF$ with $r(\cE)<r(\cF)$. We say that $\cF$ is $\omega$-\emph{slope-polystable} if $\cF=\cF_1\oplus\ldots\oplus \cF_a$ where each $\cF_i$ is $\omega$-slope stable, and $\mu_{\omega}(\cF_1)=\ldots=\mu_{\omega}(\cF_a)$. 

Let $X$ be a complex manifold. 
A  $\PP^r$-bundle on $X$ is a morphism of complex manifolds $\cP\to X$ which locally (in the classical topology)  is isomorphic to the projection $\PP^r\times U\to U$, where $U\subset X$ is open. If $\cF$ is a vector bundle of rank $r+1$ on $X$ then the projectivization $\PP(\cF)$ is a  $\PP^r$-bundle on $X$. 
\begin{dfn}
Let $X$ be a HK manifold,  with a chosen  K\"ahler class $\omega$. Let $f\colon\cX(\omega)\to\PP^1$ be the twistor family associated to $(X,\omega)$. 
\begin{enumerate}
\item
A  vector bundle $\cF$ on $X$ is $\omega$-\emph{hyperholomorphic} if there exists a holomorphic vector bundle $\sF$ on $\cX(\omega)$ such that $\sF_{|X_0}\cong\cF$.
\item
A  $\PP^r$-bundle $\cP$ on $X$ is $\omega$-\emph{projectively hyperholomorphic} if there exists a holomorphic $\PP^r$-bundle $\sP$ on 
$\cX(\omega)$ such that $\sP_{|X_0}\cong\cP$.
\item
A  vector bundle $\cF$ on $X$ is $\omega$-\emph{projectively hyperholomorphic} if $\PP(\cF)$ is projectively hyperholomorphic.
\end{enumerate}
\end{dfn}
Let $\cF$ be an $\omega$-hyperholomorphic vector bundle  on $X$.  Then the Chern classes $c_p(\cF)\in H^{p,p}_{\ZZ}(X)$ remain  Hodge classes on all twistor deformations $X_t\coloneqq f^{-1}(t)$. In other words, a necessary condition for $\cF$ to be  hyperholomorphic is that the Chern classes of $\cF$ remain Hodge classes on all $\omega$-twistor deformations. There is a similar necessary condition for a $\PP^r$-bundle  $\pi\colon \cP\to X$ to be  $\omega$-projectively hyperholomorphic. In fact suppose that $\Pi\colon \sP\to\cX(\omega)$ is the extension of $\cP$ to the twistor family. Then  the push-forward 
$\Pi_{*}\Theta_{\sP/\cX}$ of the vertical tangent bundle  is a vector bundle on $\cX(\omega)$ which extends  the push-forward  
$\pi_{*}\Theta_{\cP/\cX}$; hence the Chern classes of $\pi_{*}\Theta_{\cP/\cX}$ remain Hodge classes on all $\omega$-twistor deformations $X_t$. 
\begin{expl}\label{expl:discetang}
Let  $\cF$ be a vector bundle on $X$, and let $\cP=\PP(\cF)$. Then $\pi_{*}\Theta_{\cP/\cX}\cong End^0\cF$ is the vector bundle of traceless endomorphisms of $\cF$. In particular we get that $c_2(\pi_{*}\Theta_{\cP/\cX})=\Delta(\cF)$. 
\end{expl}
A Theorem of Verbitsky states that if $\cF$ is  $\omega$-slope-polystable, then the necessary condition stated above is also sufficient for $\cF$ to be $\omega$-hyperholomorphic. Actually one needs only to check the Hodge hypothesis for $c_1(\cF)$, $c_2(\cF)$. A similar result holds for 
$\PP^r$-bundles. 
\begin{thm}[Verbitsky~\cite{verb:iperolom}]\label{thm:iperolom}
Let $X$ be a HK manifold, with K\"ahler class $\omega$, and let $f\colon \cX(\omega)\to\PP^1$ be the twistor family associated to $\omega$. 
\begin{enumerate}
\item
Let $\cF$ be an $\omega$-slope-polystable vector bundle  on $X$. If the Chern classes $c_1(\cF),c_2(\cF)$ remain  Hodge classes on 
all $\omega$-twistor deformations of $X$ then $\cF$ is $\omega$-hyperholomorphic.
\item
Let $\cP$ be a $\PP^r$-bundle  on $X$ such that  the vector bundle $\pi_{*}\Theta_{\cP/\cX}$ on $X$ is $\omega$-slope-polystable. If the Chern class $c_2(\pi_{*}\Theta_{\cP/\cX})$ remains a  Hodge class on all $\omega$-twistor deformations of $X$ then $\cP$ is $\omega$-hyperholomorphic.
\item
Let $\cF$ be an $\omega$-slope-polystable  vector bundle  on $X$.  If the discriminant $\Delta(\cF)$ remains a  Hodge class on 
all $\omega$-twistor deformations of $X$ then $\cF$ is
 projectively hyperholomorphic.
\end{enumerate}
\end{thm}
Below we give a corollary of Theorem~\ref{thm:iperolom} which gives answers to some of the questions formulated in Subsection~\ref{subsec:defasci} when the sheaf is projectively hyperholomorphic and poly-slope-stable. 
First  note that, given ${\mathsf a}\in\RR_{+}$, it makes sense to define ${\mathsf a}$-generic K\"ahler class by analogy with what was done in Subsubsection~\ref{subsubsec:muricamere}, see~\cite[Def.~5.3]{og:highdim}. 
\begin{crl}[Prop.~7.2 in~\cite{og:highdim}]
Let $X$ be a HK manifold. Let $\omega$ be a K\"ahler class on $X$ which is ${\mathsf a}$-generic. 
Let $\cF$ be an $\omega$-poly-slope-stable projectively hyperholomorphic vector bundle  on $X$.  Suppose that ${\mathsf a}(\cF)\le {\mathsf a}$ and $\Delta(\cF)$ remains a  Hodge class for all deformations of $X$. Then the maps $\Def(X,\cF) \to \Def(X,\det\cF)$ and $\Def(\PP(\cF)) \to \Def(X)$ are surjective.
\end{crl}
Below we state a result of Verbitsky which, like the corollary above, suggests that slope-stable projectively hyperholomorphic vector bundles behave like stable sheaves on $K3$ surfaces.
\begin{thm}[Verbitsky]
Let $X$ be a HK manifold, and let $\omega$ be a K\"ahler class on $X$. 
Let $\cF$ be an $\omega$-poly-slope-stable projectively hyperholomorphic vector bundle  on $X$. The skew-symmetric pairing
\begin{equation}
\begin{matrix}
H^1(X,End\cF)\times H^1(X,End\cF) & \lra & H^2(X,\cO_X)\cong\CC \\
(\alpha,\beta) & \mapsto & \Tr(\alpha\cup\beta)
\end{matrix}
\end{equation}
is non degenerate.
\end{thm}
\begin{proof}
See~\cite[Sect.~9]{verb:iperolom}, and also~\cite[Sect.~7]{bottini:og10asmodspace}.
\end{proof}
It follows from the theorem above that if a moduli space of stable projectively hyperholomorphic vector bundles  on $X$ is smooth, then it has a holomorphic symplectic form.

\subsection{Atomic sheaves}\label{subsec:democrito}
\setcounter{equation}{0}
The definition of atomic sheaf (and complex) was introduced in~\cite{beckmann:atomic,markman:rank1obs}. We recall its definition, and we give a few examples. 
First we introduce the (rational) Looijenga-Lunts-Verbitsky (LLV) Lie subalgebra $\gotg(X)\subset \gotg\gotl(H(X;\QQ))$: it is generated by all the 
$\gots\gotl_2$-triples $e,h,f$ for which $e\in H^2(X;\QQ)$ has the Hard Lefschetz property, i.e.~multiplication by $e^k$ defines an isomorphism
$H^{n-k}(X;\QQ)\xrightarrow{\sim} H^{n+k}(X;\QQ)$ (here $2n=\dim X$). In~\cite{looijlunts:LLV,verbcohk} (see~\cite{soldato:hodgehk} for the version over $\QQ$) it is proved that $\gotg(X)\cong\gots\goto(4,b_2(X)-2)$. More precisely, let 
\begin{equation}\label{retmukraz}
\wt{H}(X;\QQ)\coloneqq \QQ\alpha\oplus H^2(X;\QQ)\oplus \QQ\beta
\end{equation}
with bilinear symmetric form $\wt{b}_X$ defined  by declaring that the decomposition in~\eqref{retmukraz} is orthogonal,  the restriction of 
$\wt{b}_X$ to $H^2(X;\QQ)$ equals the BBF bilinear form $b_X$, $\alpha,\beta$ are isotropic for $\wt{b}_X$, and $\wt{b}_X(\alpha,\beta)=-1$. 
The $\QQ$ vector space with bilinear symmetric form $\wt{b}_X$, or equivalently the associated quadratic form 
$\wt{q}_X$, is known as the \emph{rational Mukai lattice} of $X$.
If $X$ is a $K3$ surface, then $\wt{H}(X;\QQ)=H(X;\ZZ)\otimes\QQ$,  with bilinear form given by the Mukai pairing, and $\alpha,\beta$ are generators of $H^0(X;\ZZ),H^4(X;\ZZ)$ rspectively. The Theorem of Looijenga-Lunts-Verbitsky gives an  isomorphism
\begin{equation}\label{elle2vu}
\gotg(X)\cong \gots\goto(\wt{H}(X;\QQ)).
\end{equation}
I suggest reading~\cite{gklr:algebrallv} if you wish to see $\gotg(X)$ in action - literally.

The last ingredient needed to define atomicity is the \emph{Mukai vector}
\begin{equation}
v(F)\coloneqq \ch(F)\cdot\sqrt{\td(X)}
\end{equation}
of an object $F$ of the bounded derived category $D^b(X)$ of $X$. Since $\gotg(X)\subset \gotg\gotl(H(X;\QQ))$ the annihilator 
$\Ann_{\gotg(X)}(v(F))\subset \gotg(X)$ makes sense. Let $\wt{v}\in \wt{H}(X;\QQ)$:  
 then  $\gotg(X)$ acts on $\wt{H}(X;\QQ)$ by~\eqref{elle2vu}, and hence the annihilator 
$\Ann_{\gotg(X)}(\wt{v}(F))\subset \gotg(X)$ makes sense. 
\begin{dfn}\label{dfn:eccoatomico}
Let $X$ be   a projective hyperk\"ahler manifold. An object $F\in D^b(X)$ is \emph{atomic} if there exists a non-zero $\wt{v}(F)\in \wt{H}(X;\QQ)$ such that 
\begin{equation}
\Ann_{\gotg(X)}(\wt{v}(F))=\Ann_{\gotg(X)}(v(F)),
\end{equation}
\end{dfn}
\begin{expl}
A non-zero sheaf $\cF$ on a $K3$ surface is atomic. Note that we may set $\wt{v}(\cF)=v(\cF)$. 
\end{expl}
\begin{prp}[Subsec.~5.2 in~\cite{beckmann:atomic}]
Let $X$ be   a projective hyperk\"ahler manifold, with a K\"ahler class $\omega$, and let $\cF$ be a sheaf on $X$.
\begin{enumerate}
\item
If $\cF$ is  atomic torsion-free, then it is   modular.
\item
If $\cF$ is  atomic, locally-free, and $\omega$ poly-slope-stable (see Subsubsection~\ref{subsubsec:iperolom}), then it is $\omega$-projectively hyperholomorphic.
\end{enumerate}
\end{prp}
\begin{expl}
Let $X$ be   a projective hyperk\"ahler manifold of the known deformation classes (Type $K3^{[n]}$, $\Kum_n$, OG10 or OG6) of dimension greater than $2$. Then the tangent bundle $\Theta_X$ is not atomic, see~\cite[Prop.~8.3]{beckmann:atomic}. Note that the tangent bundle of a  hyperk\"ahler manifold is $\omega$-slope-stable for any K\"ahler class $\omega$, and hence also hyperholomorphic by Theorem~\ref{thm:iperolom}. 
\end{expl}
\begin{expl}
Let $S$ be a $K3$ surface.  Let $\cE_1,\cE_2$ be  locally-free sheaves on $S$, and assume that the equalities in~\eqref{giusteipotesi} hold. 
 Then the vector bundle $\cG(\cE_1,\cE_2)$ is atomic if and only if $v(\cE_1)^2=v(\cE_1)^2=0$, see~\cite[Prop.~4.1]{og:highdim}. 
\end{expl}
\begin{rmk}\label{rmk:equivderatomico}
An important property of atomicity is the following: let $\Phi\colon D^b(X)\cong D^b(Y)$ be a derived equivalence between projective HK manifolds. An object $F\in D^b(X)$ is atomic if and only if $\Phi(F)$ is atomic.
\end{rmk}
We mention below a conjecture of Beckmann which, if confirmed, would give a positive answer to question (Q2) in 
Subsection~\ref{subsec:defasci} for stable atomic vector bundles.
\begin{cnj}[Conj.~B in~\cite{beckmann:atomic}]\label{cnj:atomicoliscio}
Let $\cF$ be an $h$-slope stable atomic vector bundle on a polarized (projective) HK manifold $(X,h)$. Then the deformation space $\Def(X)$ is smooth.
\end{cnj}
\section{Moduli of semistable sheaves  on higher-dimensional hyperk\"ahler varieties}\label{sec:alfinteoremi}
\subsection{Existence and unicity results for rigid vector bundles on  hyperk\"ahler varieties}
\setcounter{equation}{0}
We state  results of~\cite{og:fascimod,og:fascik3n}, and we sketch their proofs. 

\subsubsection{Statement of results, and examples}
Let $X$ be a HK manifold. We recall that the \emph{divisibility} of a non zero 
$\alpha\in H^2(X;\ZZ)$ is the positive generator of the ideal $\{b_X(\alpha,\beta)\mid \beta\in H^2(X;\ZZ)\}$, where $b_X$ is the BBF bilinear symmetric form. We denote the divisibility of 
$\alpha$ by $\divisore(\alpha)$.  If $X$ is a $K3$ surface  the divisibility of 
$\alpha$ is equal to the maximum $l\in\NN_{+}$ dividing $\alpha$ (i.e.~such that $\alpha=l\alpha_0$ with $\alpha_0\in H^2(X;\ZZ)$), because $b_X$ is unimodular. If $\dim X>2$ then $b_X$ is not necessarily unimodular, in fact in all of the known deformations classes it is not unimodular. It follows that 
$\divisore(\alpha)$ is a multiple of
 the maximum $l\in\NN_{+}$ dividing $\alpha$, but it need not be equal. 
\begin{expl}
 Let $X$ be a HK manifold of Type $K3^{[n]}$ with $n\ge 2$. Then there exists a primitive (i.e.~divisible only by $\pm 1$) $\alpha\in H^2(X;\ZZ)$ such that 
 $q_X(\alpha)=e$ and $\divisore(\alpha)=2$ if and only if $e\equiv -2(n-1)\pmod{8}$. This follows from~\eqref {bbfk3n}. Of course there exists a primitive  
 $\alpha\in H^2(X;\ZZ)$ such that 
 $q_X(\alpha)=e$  and $\divisore(\alpha)=1$ if and only $e$ is even.
\end{expl}
A polarized HK manifold $(X,h)$ consists of a  HK manifold $X$ and a   primitive ample divisor class $h$. The \emph{degree of the polarization} is $q_X(h)$,  
and the \emph{divisibility of the polarization} is $\divisore(h)$. The degree of a polarization is positive.
\begin{expl}
Let $(n,e)$ be positive integers  with 
$e$  even. There exist polarized manifolds  $(X,h)$ of Type $K3^{[n]}$ of degree $e$ and divisibility $1$, and the moduli space of such $(X,h)$ is irreducible (of dimension $19$ if $n=1$, of dimension $20$ if $n\ge 2$). 
Let $(n,e)$ be positive integers  with 
 with $n\ge 2$ and 
$e\equiv -2(n-1)\pmod{8}$. There exist polarized manifolds   $(X,h)$ of Type $K3^{[n]}$ of degree $e$ and divisibility $2$, and the moduli space of such 
$(X,h)$ is irreducible (of dimension  $20$). For these results see~\cite[Thm.~3.5]{debarre:hksurvey}.
\end{expl}
The result below is a special case of~\cite[Thm.~1.1]{og:fascik3n} (see~\cite[Subsubsect.~1.3.1]{og:fascik3n}).
\begin{thm}[O'Grady]\label{thm:rigidisuk3n}
Let $n,r_0,e\in\NN_{+}$ such that $n\ge 2$ and 
\begin{equation}\label{econ}
e\equiv
\begin{cases}
4(n-1)r_0-2n-6 \pmod{8r_0} & \text{if $r_0\equiv 0 \pmod{4}$,} \\
\frac{1}{2}((n-1)r_0-n-3) \pmod{2r_0} & \text{if $r_0\equiv 1 \pmod{4}$,} \\
-2n-6 \pmod{8r_0}  & \text{if $r_0\equiv 2 \pmod{4}$,} \\
-\frac{1}{2}((n-1)r_0+n+3) \pmod{2r_0}  & \text{if $r_0\equiv 3 \pmod{4}$.}
\end{cases}
\end{equation}
Let $i\in\{1,2\}$ be such that $i\equiv r_0\pmod{2}$. Let $(X,h)$ be a general polarized HK manifold of Type $K3^{[n]}$ with $q_X(h)=e$ and $\divisore(h)=i$. 
There exists one and only one (up to isomorphism) $h$ slope-stable vector bundle $\cF$ on $X$ such that
 \begin{equation}\label{chernvbe}
r(\cF)=r_0^n,\quad c_1(\cF)=\frac{r_0^{n-1}}{i} h,\quad \Delta(\cF)  =  \frac{r_0^{2n-2}(r_0^2-1)}{12}c_2(X). 
\end{equation}
Moreover  $H^p(X,End^0(\cF))=0$ for all $p$. 
\end{thm}
We outline the proof of Theorem~\ref{thm:rigidisuk3n} in Subsubsections~\ref{subsubsec:esiste} and~\ref{subsubsec:unico}. Here we discuss two notable choices of $n,r_0,e$. 
\begin{expl}\label{expl:cubicadim4}
 A  general polarized HK manifold  $(X,h)$ of Type $K3^{[2]}$ of degree $6$ and divisibility $2$ is isomorphic to the variety of lines $F(Y)$ on a general cubic hypersurface  $Y\subset\PP^5$ polarized by the Pl\"ucker embedding. The vector bundle $\cF$  on $F(Y)$ of 
Theorem~\ref{thm:rigidisuk3n} (with $r_0=2$) is isomorphic to the restriction  of the tautological quotient vector bundle on $\Gr(2,\CC^6)$, 
see~\cite[Sect.~8]{og:fascimod}.  
\end{expl}
\begin{expl}\label{expl:debvoi4}
 A general polarized HK manifold  $(X,h)$ of Type $K3^{[2]}$ of degree $22$ and divisibility $2$ is isomorphic to a Debarre-Voisin variety 
\begin{equation}\label{eqDV}
X_\sigma:=\{[W]\in \Gr(6,V_{10}) \mid \sigma\vert_{ W}=0\},\qquad \sigma\in \bigwedge^3 V_{10}^\vee.
\end{equation}
 The vector bundle $\cF$ on $X_\sigma$ of Theorem~\ref{thm:rigidisuk3n} (with $r_0=2$) is isomorphic to the restriction  of the tautological quotient vector bundle on 
 $\Gr(6,V_{10})$, see~\cite[Sect.~8]{og:fascimod}.  For an application of Theorem~\ref{thm:rigidisuk3n} to the 
period map of Debarre-Voisin varieties see~\cite[Thm.~1.8]{og:fascimod}.
\end{expl}
\subsubsection{Proof of existence}\label{subsubsec:esiste}
We prove existence by considering the vector bundle $\cE[n]^{+}$ on $S^{[n]}$ associated to a spherical vector bundle $\cE$ on a $K3$ surface $S$, see 
Example~\ref{expl:fibvettrigidi}. We suppose that 
$S$ has an elliptic fibration   $S\to\PP^1$.    It follows (see Example~\ref{expl:essennelagr}) that $S^{[n]}$ has the Lagrangian fibration 
\begin{equation}\label{fiblagk3n}
\begin{matrix}
S^{[n]} & \xrightarrow{\pi} & (\PP^1)^{(n)}\cong\PP^n \\
[Z] & \mapsto & \sum_{x\in\PP^1}\left(\sum_{y\in f^{-1}(x)}l(\cO_{Z,y})\right)x
\end{matrix}
\end{equation}
We set
\begin{equation}\label{eccoeffe}
f\coloneqq \pi^{*}c_1(\cO_{\PP^n}(1)). 
\end{equation}
The next
result, which follows from surjectivity of the period map for $K3$ surfaces, is needed to produce the spherical vector bundle $\cE$ on  $S$.
\begin{clm}\label{clm:eccoell}
Let $m_0,d_0$ be positive natural numbers. There exist $K3$ surfaces $S$  with an elliptic fibration $S\to \PP^1$ such that, letting $C$ be  an elliptic fiber, we have
\begin{equation}\label{neronsevero}
\NS(S)=\ZZ[D]\oplus\ZZ[C], \quad
D\cdot D=2m_0,\quad D\cdot C=d_0.
\end{equation}
\end{clm}
For the following result see~\cite[Prop.~6.2]{og:fascimod}. 
\begin{prp}\label{prp:rigsuk}
Let $m_0,r_0,s_0,d_0\in\NN_{+}$. Suppose  that  
\begin{enumerate}
\item[(a)]
$m_0=r_0s_0-1$,
\item[(b)]
$d_0$ is  coprime to $r_0$, 
\item[(c)]
$4d_0>(2m_0+1)r_0^2(r_0^2-1)$.
\end{enumerate}
 Let $S$ be an elliptic $K3$ surface as in Claim~\ref{clm:eccoell}.
 There exists a vector bundle $\cE$ on $S$ such that
 \begin{enumerate}
\item
 $v(\cE)=(r_0,D,s_0)$,
\item
$\cE$ is spherical, i.e.~$h^0(S,End(\cE))=h^2(S,End(\cE))=1$ and $h^1(S,End(\cE))=0$,
\item
$\cE$  is $h_S$ slope-stable for any polarization $h_S$ of $S$, 
\item
the restriction of $\cE$ to \emph{every} fiber of  the  elliptic fibration $S\to \PP^1$ is slope-stable. 
\end{enumerate}
\end{prp}
\begin{rmk}
In Item~(4) we refer to stability of  the restriction of $\cE$ to  fibers of  the  elliptic fibration $S\to \PP^1$. A finite number of these fibers are singular irreducible curves. Slope-stability of a sheaf on such a fiber is defined as in the case of a smooth curve and is independent of the choice of a polarization.
\end{rmk}
Let  $n,r_0,e\in\NN_{+}$ be as in Theorem~\ref{thm:rigidisuk3n}. The first step of the proof consists in showing that there exist a $K3$ surface
 $S$   as in Claim~\ref{clm:eccoell} and  a vector bundle $\cE$ on $S$ as in Proposition~\ref{prp:rigsuk}, such that rank, first Chern class and discriminant of   the vector bundle $\cE[n]^{+}$ on $S^{[n]}$ are as those of the vector bundle $\cF$ in~\eqref{chernvbe}.
Set
\begin{equation*}
\ov{e}:=
\begin{cases}
\text{$e$ if $r_0$ is even,} \\
 \text{$4e$  if $r_0$ is odd,}
\end{cases}
\end{equation*}
and let
\begin{equation}\label{chartroux}
s_0:=\frac{\ov{e}+ (2n-2)(r_0-1)^2+8}{8r_0},\quad m_0:=\frac{\ov{e}+ 2(n-1)(r_0-1)^2}{8}. 
\end{equation}
Then $s_0, m_0$ are  integers by~\eqref{econ}. 
A straightforward computation gives that 
\begin{equation}\label{stanco}
 m_0=r_0s_0-1. 
\end{equation}
Let $S$ be a $K3$ surface as in Claim~\ref{clm:eccoell}, where $m_0$ is as in~\eqref{chartroux}, and $d_0$ is an integer  coprime to $r_0$  such that the inequality in Item~(c) of Proposition~\ref{prp:rigsuk} holds. By Proposition~\ref{prp:rigsuk} 
there exists a  vector bundle $\cE$ on $S$ such that 
\begin{equation}\label{lezozzone}
v(\cE)=(r_0,D,s_0)
\end{equation}
and Items (2)-(4) of that  proposition hold.  
\begin{clm}\label{clm:classigiuste}
Keep notation and hypotheses as above, in particular $\cE$ is the vector bundle on $S$ such that the equation in~\eqref{lezozzone} 
and Items (2)-(4) of  Proposition~\ref{prp:rigsuk} hold. 
 Then 
 \begin{equation}\label{robertbyron}
r(\cE[n]^{+})=r_0^n,\quad c_1(\cE[n]^{+})=\frac{r_0^{n-1}}{i} h,\quad \Delta(\cE[n]^{+})  =  
\frac{r_0^{2n-2}(r_0^2-1)}{12}c_2(S^{[n]}),
\end{equation}
where $h\in \NS(S^{[n]})$ is primitive, $q(h)=e$ and $\divisore(h)=i$.
\end{clm}
\begin{proof}
The formula for the rank of $\cE[n]^{+}$ is straightforward. The statements about $c_1(\cE[n]^{+})$ and $\Delta(\cE[n]^{+})$ follow 
from~\cite[Prop.~3.2]{og:fascik3n}.
\end{proof}
\begin{rmk}
We have used the letter $h$ for the divisor class appearing in~\eqref{robertbyron}. This might suggest the \emph{wrong} impression that it is an ample class.
\end{rmk}
A divisor class $h$ on a product $X_1\times \ldots\times X_n$ of projective varieties is of \emph{product type} if 
$h=p_{X_1}^{*}(h_1)+\ldots+p_{X_n}^{*}(h_n)$ where, for $i\in\{1,\ldots,n\}$, 
$h_i$ is a divisor class on $X_i$ and  $p_i\colon
X_1\times \ldots\times X_n\to X_i$ is the projection. Note that such an $h$ is ample if and only  each $h_i$ is  ample.
\begin{prp}\label{prp:restrizstab}
Keep notation and hypotheses as  in Claim~\ref{clm:classigiuste}. Let $x_1,\ldots,x_n\in\PP^1$ be pairwise distinct.
The restriction of $\cE[n]^{+}$ to $\pi^{-1}(x_1+\ldots+x_n)\cong C_{x_1}\times\ldots\times C_{x_n}$ 
is slope stable for any polarization of product type.
\end{prp}
\begin{proof}
Since  $x_1,\ldots,x_n\in\PP^1$ are pairwise distinct we have the isomorphisms
\begin{equation*}
\pi^{-1}(x_1+\ldots+ x_n)\cong C_{x_1}\times\ldots\times C_{x_n},
\end{equation*}
and
\begin{equation*}
\cE[n]^{+}_{\pi^{-1}(x_1+\ldots+ x_n)}\cong \left(\cE_{|C_{x_1}}\right)\boxtimes\ldots\times \left(\cE_{|C_{x_n}}\right).
\end{equation*}
It follows that the restriction of $\cE[n]^{+}$ to $\{p_1\}\times \ldots \times \ldots\{p_{i-1}\}\times C_{x_i}\times\{p_{i+1}\}\times\ldots\times \ldots\{p_{n}\}$ is a direct sum of copies of the vector bundle $\cE_{C_{x_i}}$, which is stable
by Proposition~\ref{prp:rigsuk}. The proposition follows easily from this, see~\cite[Lemma~4.7]{og:fascik3n}. 
\end{proof}
Keep notation and assumptions as above. Suppose for a moment that $c_1(\cE[n]^{+})=r_0^{n-1}h/i$ (notation as in~\eqref{robertbyron}) is ample, and that $h$ is 
$a(\cE[n]^{+})$-suitable for $f$, where $f$ is given by~\eqref{eccoeffe}. Then  Proposition~\ref{prp:restrizstab} and 
Corollary~\ref{crl:lagstab} imply that $\cE[n]^{+}$ is $h$ slope stable. Moreover the map $\Def(\cE[n]^{+})\to \Def(S^{[n]},h)$ is surjective because  of Theorem~\ref{thm:defcoppia} and the equality (see~\eqref{extpiu:tre})
\begin{equation}
\ext^2_{S^{[n]}}(\cE[n]^{+},\cE[n]^{+})=1. 
\end{equation}
Hence  $\cE[n]^{+}$ extends to a vector bundle on the general polarized deformation of 
$(S^{[n]},h)$, and the extended vector bundle is slope stable by openness of stability. This gives the existence half of Theorem~\ref{thm:rigidisuk3n}  because the relevant moduli space is irreducible 
(warning: irreducibility holds because of our choice of degree and divisibility of $h$, it does not hold unconditionally). 

The  argument given  above does not go through as stated because in 
general $h$ is not ample, and/or it is not $a(\cE[n]^{+})$-suitable for $f$. One gets around this obstacle arguing as follows. Let 
$X$ be a very generic (small) deformation of $S^{[2]}$ keeping the classes $h,f$ algebraic. As mentioned above the vector bundle $\cE[n]^{+}$ extends to a vector bundle $\cF_X$ on $X$. Since $X$ is a very generic we have $\NS(X)=\ZZ h_X\oplus\ZZ f_X$ where $h_X, f_X$ are the classes obtained by parallel transport from $h,f$. The Lagrangian fibration $\pi\colon S^{[n]}\to\PP^n$ extends to a Lagrangian fibration 
$\pi_X\colon X\to \PP^n$, and $f_X=\pi_X^{*}c_1(\cO_{\PP^n}(1))$. By openness of stability the 
restriction of $\cF_X$ to a general fiber of $\pi_X$ is slope stable.
If $d_0\gg 0$ the class $h_X$ is ample and $a(\cE[n]^{+})$-suitable for $f_X$. In fact if $\gamma\in(\ZZ h\oplus\ZZ f)$ and $q_X(\gamma)<0$ then 
$q_X(\gamma)\le 2d_0/(1+e)$ (see~\cite[Lmm.~4.3]{og:fascimod}): this implies that $h$ is ample by the known numerical description of the ample cone of HK manifold of Type $K3^{[n]}$, and  that $h$ is  $a(\cE[n]^{+})$-suitable for $f$ because there is a single 
open $a(\cE[n]^{+})$-chamber (because $d_0\gg 0$). Thus $\cF_X$ is $h_X$ slope-stable because its restriction  to a general fiber of $\pi_X$ is slope stable.
\subsubsection{Proof of uniqueness}\label{subsubsec:unico}
This is the most subtle half of the proof. The main steps are the following:
\begin{enumerate}
\item
Let $\cK_e^i(2n)$ be the moduli space of polarized HK manifolds $(Y,h_Y)$ of Type $K3^{[n]}$ with  $q_{Y}(h_Y)=e$ and $\divisore(h_Y)=i$. Let  
$NL(d_0)\subset\cK_e^i(2n)$ be the Noether-Lefschetz divisor parametrizing deformations of $(X,h_X,f_X)$, where $X,h_X,f_X$ are obtained from $S^{[n]},h,f$ as described at the end of Subsubsection~\ref{subsubsec:esiste}. It suffices to prove that for $d_0\gg 0$ and for general 
$[(Y,h_Y)]\in NL(d_0)$ there exists a unique slope stable vector bundle with rank, first Chern class and discriminant given by~\eqref{chernvbe} 
with $X=Y$ and $h=h_Y$. In fact this suffices because by~\cite{maruyama:fascilimitati} there exists a quasi-projective  
${\mathfrak M}\to \cK_e^i(2n)$ (more precisely ${\mathfrak M}$ maps to the stack of polarized...) which is a relative moduli space of slope stable vector bundles with the relevant 
rank, first Chern class and discriminant, and the union of the Noether-Lefschetz divisors $NL(d_0)$ for $d_0\gg 0$ is dense in $\cK_e^i(2n)$.
\item
Let $X$ and $\cF_X$ be as in  the end of Subsubsection~\ref{subsubsec:esiste}, and let $\pi_X\colon X\to\PP^n$ be the Lagrangian fibration. A key observation is that the (open) subset  $\cU(\cF_X)\subset\PP^n$ parametrizing $t$ such that the restriction of $\cF_X$ to $\pi_X^{-1}(t)$ is slope stable is big, where \lq\lq big\rq\rq\ means that its complement has codimension at least $2$. By openness of slope-stability it follows that for a general $[(Y,h_Y)]\in NL(d_0)$ there exists an $h_Y$ slope-stable  vector bundle $\cF_Y$ on $Y$ for which $\cU(\cF_Y)$ is big. Let 
$t\in\cU(\cF_Y)$ be such that $\pi^{-1}(t)$ is smooth: by Proposition~\ref{prp:ressemi} the restriction of $\cF_Y$ to  $\pi^{-1}(t)$ is semi-homogeneous.  
\item
Let $[(Y,h_Y)]\in NL(d_0)$ be a general point, and let $\cG$ be an $h_Y$ slope-stable vector bundle on $Y$ with rank, first Chern class and discriminant given by~\eqref{chernvbe} 
(with $X=Y$, $h=h_Y$ and $\cF=\cG$). By Proposition~\ref{prp:lagstab} the restriction of $\cG$ to a general Lagrangian fibre $\pi_Y^{-1}(t)$ is 
$\theta_t$ slope-semi-stable. One proves that in fact $\cG_{|\pi_Y^{-1}(t)}$ is $\theta_t$ slope-stable, and hence  by Proposition~\ref{prp:ressemi} it is semi-homogeneous.  
\item
Let $[(Y,h_Y)]\in NL(d_0)$ be a general point. Let $\cF_Y,\cG$ be vector bundles  as in Items~(3) and~(4). We need to  prove that $\cF_Y\cong\cG$. Let 
$t\in\PP^n$ be a general point (in particular $\pi^{-1}(t)$ is smooth). By Items~(3), (4) the restrictions of $\cF_Y,\cG$ to $\pi^{-1}(t)$ are simple semi-homogeneous vector bundles satisfying the hypothesis in Item~(b) of Theorem~\ref{thm:vivasemihom}, and hence they differ by the twist of a torsion line-bundle. A monodromy argument proves that they are actually isomorphic. 
The last step is to prove that the locus of $t\in\PP^n$ for which $\cF_{Y|\pi^{-1}(t)}\cong\cG_{|\pi^{-1}(t)}$ is big, and is based 
on~\cite[Lmm.~7.5]{og:fascimod}. It follows that  $\cF_Y\cong\cG$. 
\end{enumerate}
\subsubsection{Comparison with unicity for rigid stable vector bundles on $K3$ surfaces}
Uniqueness of a slope stable rigid vector bundle  on a $K3$ surface with assigned Mukai vector follows from Mukai's simple cohomological  argument,  see Comment~\ref{cmm:verfica}. One cannot  extend such an argument to prove the uniqueness half of Theorem~\ref{thm:rigidisuk3n}, because we have no control of $\ext^{2p}_X(\cE,\cF)$ for $0<p<n$, where $\cE,\cF$ are stable vector bundles with the assigned rank, first Chern class and discriminant. The proof that we have outlined is the analogue of  a proof of unicity of a rigid stable vector bundle  on an elliptic $K3$ surface $X$ (with assigned Mukai vector) such that $r(\cF)$ and $c_1(\cF)\cdot C$ are coprime ($C$ is an elliptic fiber) which goes as follows. If $\cE,\cF$ are two such vector bundles on $X$, then the restrictions of $\cE,\cF$ to every elliptic fiber are stable (see Proposition~\ref{prp:rigsuk}), and hence are isomorphic by Atiyah's classification of indecomposable vector bundles on elliptic curves (which extends to genus $1$ curvs with a node). Since $c_1(\cE)=c_1(\cF)$ it follows that $\cE$ and $\cF$ are isomorphic.
\subsubsection{Scope of Theorem~\ref{thm:rigidisuk3n}}
How far are we from a classification of all rigid stable vector bundles $\cF$ on a HK manifold $X$ of $K3^{[n]}$ type with the property that 
$\PP(\cF)$ extends to all (small) deformations of $X$? (Abusing notation let us agree to say that such an $\cF$ is 
\emph{projectively hyperholomorphic}.)
 At a first glance the ranks and first Chern classes  in Theorem~\ref{thm:rigidisuk3n} are surprisingly special. A result of Mukai\cite[Cor.~7.12]{muksemi}  on ranks and first Chern classes of semi-homogeneous vector bundles on principally polarized abelian varieties justifies the shape of those ranks and first Chern classes. In fact if $\cF$ is a rigid  slope-stable vector bundle on a HK manifold $X$ of Type $K3^{[n]}$ which extends to a vector bundle $\cE$ on a  deformation $Y$ of $X$ which has a Lagrangian fibration $f\colon Y\to\PP^n$, then the smooth fibers of $f$ are (torsors over) principally polarized abelian varieties. If for general $t\in\PP^n$ the restriction $\cE_{|f^{-1}(t)}$ is slope stable, then it is a simple semi-homogeneous vector bundle 
(see Proposition~\ref{prp:ressemi}), and hence Mukai's result applies. Next we comment on the discriminant
appearing in~\eqref{chernvbe}. If we assume that $\cF$ is projectively hyperholomorphic, then  $\Delta(\cF)$ remains of type $(2,2)$ for all deformations of $X$, see Example~\ref{expl:discetang}. It follows that
 $\Delta(\cF)$ is a  linear combination of $c_2(X)$ and $q_X^{\vee}$ (see  the main result in~\cite{zhang-char-form}), where $q_X^{\vee}$ is the image via the cup-product map $H^2(X)\otimes H^2(X)\to H^4(X)$ of the class in $\Hom(H^2(X)^{\vee},H^2(X))$ inverse of the BBF bilinear symmetric form $b_X\colon H^2(X)\xrightarrow{\sim} H^2(X)^{\vee}$.
   If $n\in\{2,3\}$ then $c_2(X)$ and $q_X^{\vee}$ are linearly dependent, and hence  $\Delta(\cF)$ is  a multiple of $c_2(X)$.  If $n>3$ then $c_2(X)$ and $q_X^{\vee}$ are linearly independent, hence there is no  \lq\lq a priori\rq\rq\  reason why  $\Delta(\cF)$ should be   a multiple of $c_2(X)$.  I know of no examples of stable projectively hyperholomorphic  vector bundles on HK varieties of type $K3^{[n]}$ whose discriminant is not a multiple of the second Chern class.
 In~\cite[Subsec.~1.3.2]{og:fascik3n} there are some comments relevant to the above questions\footnote{Warning: the hypotheses in~\cite[Prop.~1.2]{og:fascik3n} that $\Def(\cF)\to\Def(X,h)$ is surjective must be replaced 
by   surjectivity of the map $\Def(\PP(\cF))\to\Def(X)$.}.

\subsubsection{The Generalized Franchetta Conjecture}
Let $\cM$ be a moduli space of polarized HK manifolds (of a fixed deformation type), and let $\cX\to\cM$ be th universal polarized HK manifold (for this to make sense we have to restrict to an open dense subset of $\cM$, or replace $\cM$ by the moduli stack). The Generalized Franchetta Conjecture, see~\cite{fulatvial:genfranchetta1}, predicts that if  ${\mathfrak z}\in \CH^{*}(\cX)_{\QQ}$ 
restricts to a cohomologically trivial class on a general fiber $X_t$ of $\cX\to\cM$ then it  is actually rationally equivalent to $0$ on each $X_t$.
Theorem~\ref{thm:rigidisuk3n} provides a non trivial testing ground for the Generalized Franchetta Conjecture. 
In fact let $n,r_0,e,i$  be as in that theorem, and let  $\cK_e^i(2n)$ be the moduli stack of polarized HK manifolds $(Y,h_Y)$ of Type $K3^{[n]}$ with  $q_{Y}(h_Y)=e$ and $\divisore(h_Y)=i$. Let $\cX\to \cK_e^i(2n)$ be the tautological family of HK (polarized) varieties. 
By~\cite[Thm.~A.5]{mukai:vbtata}   there exists a quasi tautological vector bundle $\mathsf F$ on $\cX$, i.e.~a vector bundle whose restriction to a fiber $X$ of 
$\cX\to \cK_e^i(2n)$ is isomorphic to $\cF^{\oplus d}$ for some $d>0$, where $\cF$ is the vector bundle of Theorem~\ref{thm:rigidisuk3n}. According to the Generalized Franchetta Conjecture  the restriction to $\CH^2(X)_{\QQ}$  of  
$\ch_2({\mathsf F}^{\vee}\otimes{\mathsf F})\in\CH(\cX)_{\QQ}$ is equal to  $-d^2\frac{r_0^{2n-2}(r_0^2-1)}{12}c_2(X)$. In other words it predicts that the third equality in~\eqref{chernvbe} holds at the level of (rational) Chow groups.  
\subsubsection{ Analogues for other deformation types}
In~\cite{og:fascikum} we have proved an existence and unicity result for stable rigid vector bundles of rank $4$ on  general polarized HK fourfolds of Kummer type (see Example~\ref{expl:esempihk}). One motivation was to provide an explicit description of a locally complete family 
of polarized fourfolds of Kummer type analogous to that of  a  general polarized HK manifold of Type $K3^{[2]}$ of degree $6$ and divisibility $2$ as the variety of lines on a smooth cubic four-dimensional hypersurface (see Example~\ref{expl:cubicadim4}), or of a general polarized HK manifold of Type $K3^{[2]}$ of degree $22$ and divisibility $2$ as a Debarre-Voisin variety (see Example~\ref{expl:debvoi4}).  No such description is known, but see~\cite{abgr:modellikum} for progress on this problem. A more general result on existence and uniqueness of stable rigid vector bundles on  polarized HK manifolds of Kummer type is certainly within reach. 

A more challenging problem is to give results on  stable (rigid or non-rigid) vector bundles on HK manifolds of Type OG10 (see Item~(3) of Theorem~\ref{thm:vettmuknonprim}) or of Type OG6 (the \lq\lq last\rq\rq\  known deformation class of HK manifolds~\cite{og:voilaog6}). It is more challenging because for Type $K3^{[n]}$ one produces examples by relating equivariant bundles on $S^n$ to vector bundles on $S^{[n]}$ via the BKR equivalence (here $S$ is a $K3$ surface), and similarly 
for Type $\Kum_n$.  No such construction is available for HK manifolds of Type OG10 or OG6.

\subsubsection{  $\PP^{r_0^n-1}$ bundles}
Let $(X,h)$ and $\cF$ be as in Theorem~\ref{thm:rigidisuk3n}. Then $\cF$ is projectively hyperholomorphic by Theorem~\ref{thm:iperolom}, 
i.e.~the projectivization $\PP(\cF)$ extends, as $\PP^{r_0^n-1}$ bundle, to all $h$-twistor deformations of $X$. Iterating this procedure one 
gets a $\PP^{r_0^n-1}$ bundle on every HK manifold of Type $K3^{[n]}$, see~\cite[Thm.~1.4]{markman:rank1obs}. This has proved to be useful in exploring the relation between the period and index of Brauer classes on HK manifolds, see~\cite{bottinidh:period-index-hk}.

\subsection{OG10 as a moduli space of vector bundles}
\setcounter{equation}{0}
We  recall that  ten-dimensional HK manifolds of Type OG10 are deformations of symplectic desingularizations of suitable moduli spaces of sheaves on $K3$ surfaces (see Item~(3) of Theorem~\ref{thm:vettmuknonprim}), in fact they were discovered in such a guise~\cite{og:voilaog10}.  The main result of~\cite{bottini:og10asmodspace} (see also~\cite{bottini:towardsog10}) may be summarized as follows: one gets HK manifolds of Type OG10 as connected components of certain moduli spaces of slope stable vector bundles on HK manifolds  of Type $K3^{[2]}$. In fact Bottini's papers contain many other interesting results, and they are linked to the paper~\cite{lsv:og10ascompcftdjac} which describes a $20$-dimensional family of projective HK manifolds of Type OG10. We will not give any details of the proofs, except for mentioning that the stable vector bundles which appear in Bottini's papers are atomic, and they are related to rank-$1$ sheaves on Lagrangian surfaces in the variety of lines of certain cubic fourfolds via a suitable derived equivalence 
(see Remark~\ref{rmk:equivderatomico}).
\subsection{Isogenies between HK manifolds of Type $K3^{[n]}$, and the $D$-equivalence Conjecture}
\setcounter{equation}{0}

Let $(S,h)$ be a polarized  $K3$ surface. Suppose that  $v=(r,mh,s)$ is a primitive isotropic vector for $S$ with $r,s$ coprime, and that $h$ is ${\mathsf a}(v)$-generic (e.g.~if $\NS(S)=\ZZ [h]$). Then the moduli space $M\coloneqq \cM_v(S,h)$ is a $K3$ surface, and there exists a universal vector bundle $\cU$ on  $S\times M$ (because 
$\gcd(r,s)=1$). Via the BKR equivalence one constructs a vector bundle $\cU^{[n]}$ on $S^{[n]}\times M^{[n]}$. If $[Z]\in M^{[n]}$ and $Z$ is reduced, say $Z=\{p_1,\ldots,p_n\}$, then  the restriction of $\cU^{[n]}$ to $S^{[n]}\times\{[Z]\}$ is isomorphic to $\cG(\cU_{p_1},\ldots,\cU_{p_n})$ (see~\eqref{gieunoedue}), where $\cU_{p_1}\coloneqq \cU_{|S\times\{p_i\}}$ is the vector-bundle on $S$ corresponding to $p_i$.  The Fourier-Mukai transform with kernel  $\cU^{[n]}$ defines an equivalence between the (bounded) derived categories of $S^{[n]}$ and $M^{[n]}$, and moreover it defines an algebraic cycle inducing a Hodge isometry $H^2(S^{[n]};\QQ)\xrightarrow{\sim} H^2(M^{[n]};\QQ)$ which lifts the 
Hodge isometry $H^2(S;\QQ)\xrightarrow{\sim} H^2(M;\QQ)$ defined by $\cU$.
Markman~\cite{markman:hodgeisom} proves that $\cU^{[n]}$ deforms to a twisted vector bundle on couples $(X,Y)$ obtained from 
$(S^{[n]},M^{[n]})$ via \lq\lq\ diagonal twistor lines\rq\rq. From this he deduces that the analogue of the Shafarevich conjecture holds for HK manifolds of Type $K3^{[n]}$. In~\cite{mauliketalii:diequivalenza} the vector bundles $\cU^{[n]}$ are used to prove the $D$-Conjecture for (projective) HK manifolds of Type $K3^{[n]}$ (actually they prove a more general statement). The vector bundles  $\cU^{[n]}$ have been useful in 
exploring the relation between period and indx for Brauer classes on HK manifolds of Type $K3^{[n]}$, see~\cite{hmsyz:per-index-type-k3n}.
\subsection{HK manifolds of Type $K3^{[a^2+1]}$ as moduli spaces of projective bundles on a HK manifold  of Type $K3^{[2]}$}\label{subsec:lavoroinfinito}
\setcounter{equation}{0}
The starting point is the following. Let $S$ be a $K3$ surface. Let  $D$ be a divisor class on $S$.
Let $a,s_1\in\NN$, with $a>0$ such that $2as_1=D\cdot D-2$. Let
\begin{equation}
v_1\coloneqq\left(2a,D,s_1\right),\qquad v_2\coloneqq av_1-(0,0,1).
\end{equation}
Note that $v_1^2=-2$ and $v_2^2=2a^2$ (squares are with respect to Mukai's pairing). Moreover  $v_1,v_2$ are primitive.
Let $h_S$ be an ${\mathsf a}(v_2)$-generic polarization of $S$ (and hence also ${\mathsf a}(v_1)$-generic). The moduli space 
$\cM_{v_1}(S,h_S)$ is a reduced point, and  $\cM_{v_2}(S,h_S)$ is a HK variety of Type $K3^{[a^2+1]}$. We let
\begin{equation}
\cM_{v_1}(S,h_S)=\{[\cE_1]\}.
\end{equation}
Thus $\cE_1$ is a spherical vector bundle.
Let  $[\cE_2]\in \cM_{v_2}(S,h_S)$, and let $\cG(\cE_1,\cE_2)$ be the torsion-free sheaf on $S^{[2]}$ given by~\eqref{gieunoedue}. We recall that 
$\cG(\cE_1,\cE_2)$ is modular. We have 
\begin{equation}\label{sanmichele}
r(\cG(\cE_1,\cE_2)) =8a^3,\qquad \Delta(\cG(\cE_1,\cE_2)) =\frac{ 4a^6}{3} c_2(S^{[2]}),\qquad 
{\mathsf a}(\cG(\cE_1,\cE_2)) =2560 a^{12}. 
\end{equation}
For the second equality see~\eqref{discingen}, for the third one see~\cite[Expl.~2.14]{og:highdim}. In order to state a  first stability result on the sheaves $\cG(\cE_1,\cE_2)$ we need to define a divisor on $\cM_{v_2}(S,h_S)$. Let $[\cE_2]\in\cM_{v_2}(S,h_S)$. By stability of $\cE_2$ we have $\Hom_S(\cE_1,\cE_2)=0$, and hence by Serre duality $\Ext^2(\cE_2,\cE_1)=0$.
Moreover we have
\begin{equation}
\chi_S(\cE_2,\cE_1)=-\la v_2,v_1\ra=0.
\end{equation}
Hence we expect that the determinantal subscheme
\begin{equation}
\cD_{v_2}(S,h_S)\coloneqq\{[\cE_2]\in\cM_{v_2}(S,h_S) \mid \Hom_S(\cE_1,\cE_2)\not=0\}
\end{equation}
is a divisor. In fact it is a divisor. 
\begin{prp}[\cite{og:genhk}]\label{prp:stabuno}
Keep hypotheses and notation  as above, in particular $h_S$ is ${\mathsf a}(v_2)$-generic.  Let  $ h_{S^{[2]}}$ be a polarization of $S^{[2]}$ which is $2560 a^{12}$-suitable for $\bm{\mu}(h_S)$, where $\bm{\mu}$ is as in Example~\ref{expl:bbfhilb}. 
Let $[\cE_2]\in(\cM_{v_2}(S,h_S)\setminus \cD_{v_2}(S,h_S)) $.  Then 
$\cG(\cE_1,\cE_2)$ is an $h_{S^{[2]}}$ slope stable vector bundle. 
\end{prp}
\begin{cmm}
In the outline of the proof of Proposition~\ref{thm:rigidisuk3n} (see Subsubsection~\ref{subsubsec:esiste}) we showed that the vector bundles 
$\cE[n]^{+}$ are stable if $S$ is elliptic and the polarization is ${\mathsf a}(\cE[n]^{+})$-suitable for the class coming from the Lagrangian fibration $S^{[n]}\to\PP^n$. 
A similar approach towards stability of $\cG(\cE_1,\cE_2)$ runs into the problem that the restriction of $\cG(\cE_1,\cE_2)$ to a general Lagrangian fiber is slope-semistable but not stable - this is the reason why in Proposition~\ref{prp:stabuno} stability is with respect to a polarization which is suitable for $\bm{\mu}(h_S)$. I expect that non-rigid  (modular) vector bundles on Lagrangian HK manifolds never restrict  to a stable sheaf on a Lagrangian fiber. This is in stark contrast with what holds for vector bundles on $K3$ surfaces.
\end{cmm}
\begin{rmk}
With hypotheses as in Proposition~\ref{prp:stabuno}  the vector bundle $\cG(\cE_1,\cE_2)$ is projectively hyperholomorphic 
(by Theorem~\ref{thm:iperolom}), but it is not atomic, see~\cite[Prop.~4.1]{og:highdim}.
\end{rmk}
Let $ h_{S^{[2]}}$ be a polarization of $S^{[2]}$ as in Proposition~\ref{prp:stabuno}, and let $M_{a,D}(S^{[2]},h_{S^{[2]}})$ be the moduli space of 
 $ h_{S^{[2]}}$ slope-stable vector bundles $\cF$ such that
\begin{equation}\label{fintomukvect}
r(\cF) =8a^3,\qquad c_1(\cF)= 4a^2\bm{\mu}(D)-4a^3\delta_2, \qquad \Delta(F) =\frac{ 4a^6}{3} c_2(S^{[2]}).
\end{equation}
(See   Example~\ref{expl:bbfhilb} for the definition of  $\delta_2$.) Note that $M_{a,D}(S^{[2]},h_{S^{[2]}})$ is a quasi-projective scheme by~\cite{maruyama:fascilimitati}. If  $[\cE_2]\in\cM_{v_2}(S,h_S)$ then $c_1(\cG(\cE_1,\cE_2))$ equals 
the right-hand side of the middle equality in~\eqref{fintomukvect}. Hence by Proposition~\ref{prp:stabuno} we have a regular map
\begin{equation}\label{mappamod}
\begin{matrix}
\cM_{v_2}(S,h_S)\setminus \cD_{v_2}(S,h_S) & \lra & M_{a,D}(S^{[2]},h_{S^{[2]}}) \\
[\cE_2] & \mapsto & [(\cG(\cE_1,\cE_2)]
\end{matrix}
\end{equation}
The image is a locally closed subset: let $M_{a,D}(S^{[2]},h_{S^{[2]}})^{\bullet}$ be its closure
\begin{prp}[\cite{og:genhk}]\label{prp:stabdue}
Keep hypotheses and notation  as in Proposition~\ref{prp:stabuno}. The map in~\eqref{mappamod} extends to a regular bijective map
\begin{equation}\label{joanbaez}
\cM_{v_2}(S,h_S) \lra  M_{a,D}(S^{[2]},h_{S^{[2]}})^{\bullet}
\end{equation}
\end{prp}
By Proposition the normalization map of $M_{a,D}(S^{[2]},h_{S^{[2]}})^{\bullet}$ is given by~\eqref{joanbaez}. 
\begin{rmk}
If $a>1$ then $M_{a,D}(S^{[2]},h_{S^{[2]}})^{\bullet}$ is singular, i.e.~the map in~\eqref{joanbaez} is not an isomorphism. 
\end{rmk}
By considering a quasi-universal vector bundle on $S^{[2]}\times M_{a,D}(S^{[2]},h_{S^{[2]}})^{\bullet}$ one gets a linear map 
$H^2(S^{[2]};\QQ)\to H^2(M_{a,D}(S^{[2]},h_{S^{[2]}})^{\bullet};\QQ)$. 
Composing with the normalization map in~\eqref{joanbaez} one gets a Hodge isometry
\begin{equation}\label{romyschneider}
H^2(S^{[2]};\QQ)\lra H^2(\cM_{v_2}(S,h_S);\QQ).
\end{equation}
By  Theorem~\ref{thm:iperolom} the vector bundles parametrized by $M_{a,D}(S^{[2]},h_{S^{[2]}})^{\bullet}$ are projectively hyperholomorphic. 
Hence if $[\cF]\in M_{a,D}(S^{[2]},h_{S^{[2]}})^{\bullet}$ the projectivization $\PP(\cF)$ extends to all deformations of $S^{[2]}$ parametrized by the $h_{S^{[2]}}$-twistor line as slope-stable projective bundles, and also those parametrized by  $\omega$-twistor lines for $\omega$ a K\"ahler class sufficiently close to $h_{S^{[2]}}$ (this follows from Proposition~\ref{prp:mirren}). By~\cite[Cor.~10.1]{verb:iperolom} the local structure of the moduli spaces of projective bundles along the twistor lines does not change. The conclusion is that a component of a suitable moduli space of projective bundles on arbitrary (small) deformations of $S^{[2]}$ has as normalization a HK manifold of Type $K3^{[a^2+1]}$ (the deformation of 
$\cM_{v_2}(S,h_S)$). Since the map in~\eqref{romyschneider} is  a Hodge isometry the HK manifolds that we get  give all small deformations
 of $S^{[2]}$. Working with chains of twistor lines one gets that a general  HK manifold of Type $K3^{[a^2+1]}$ is the normalization of a (component of) of a  moduli space of (twisted) hyperholomorphic vector bundles on a HK manifold  of Type $K3^{[2]}$.


\end{document}